\documentclass[a4paper,twoside]{amsart}
\usepackage{a4}
\usepackage[normalem]{ulem}
\usepackage{amssymb}
\usepackage{amsmath}
\usepackage{upref}
\usepackage{bbm}
\usepackage[active]{srcltx}
\usepackage{enumerate}
\usepackage[dvipsnames]{color}
\usepackage{mathtools}

\definecolor{blau}{rgb}{0.1,0.0,0.9}

\newcounter{komcounter}
\numberwithin{komcounter}{section}

\makeatletter
\def\sectionmark#1{} 
\def\subsectionmark#1{}
\newcommand{\sectnr}{\ifnum \c@secnumdepth >\z@
                 \thesection.\hskip 1em\relax \fi}
\def\tableofcontents{\section*{Contents} 
 \@starttoc{toc}}
\makeatother
\makeatletter
\def\@biblabel#1{#1.}
\makeatother
\makeatletter
\let\Thebibliography=\thebibliography
\renewcommand{\thebibliography}[1]{\def\@mkboth##1##2{}\Thebibliography{#1}
\addcontentsline{toc}{section}{References}
\frenchspacing 
\setlength{\@topsep}{0pt}
\setlength{\itemsep}{0pt}%
\setlength{\parskip}{0pt plus 2pt}%
}
\makeatother
\makeatletter
\def\mdots@{\mathinner.\nonscript\!.%
 \ifx\next,.\else\ifx\next;.\else\ifx\next..\else
 \nonscript\!\mathinner.\fi\fi\fi}
\let\ldots\mdots@
\let\cdots\mdots@
\let\dotso\mdots@
\let\dotsb\mdots@
\let\dotsm\mdots@
\let\dotsc\mdots@
\def\vdots{\vbox{\baselineskip2.8\p@ \lineskiplimit\z@
    \kern6\p@\hbox{.}\hbox{.}\hbox{.}\kern3\p@}}
\def\ddots{\mathinner{\mkern1mu\raise8.6\p@\vbox{\kern7\p@\hbox{.}}%
    \raise5.8\p@\hbox{.}\raise3\p@\hbox{.}\mkern1mu}}
\makeatother
\makeatletter
\let\Enumerate=\enumerate
\renewcommand{\enumerate}{\Enumerate%
\setlength{\@topsep}{0pt}
\setlength{\itemsep}{0pt}%
\setlength{\parskip}{0pt plus 1pt}%
\renewcommand{\theenumi}{\textup{(\alph{enumi})}}%
\renewcommand{\labelenumi}{\theenumi}%
}
\let\endEnumerate=\endenumerate
\renewcommand{\endenumerate}{\endEnumerate\unskip}
\makeatother
\makeatletter
\def\@seccntformat#1{\csname the#1\endcsname.\quad}
\makeatother
\newcommand{\art}[6]{{\sc #1, \rm #2, \it #3 \bf #4 \rm (#5), \mbox{#6}.}}

\newcommand{\AND}{{\rm and }}
\RequirePackage{amsthm}
\newtheoremstyle{descriptive}%
  {\topsep}   
  {\topsep}   
  {\rmfamily} 
  {}          
  {\bfseries} 
  {.}         
  { }         
  {}          
\newtheoremstyle{propositional}%
  {\topsep}   
  {\topsep}   
  {\itshape}  
  {}          
  {\bfseries} 
  {.}         
  { }         
  {}          
\newtheoremstyle{remarkstyle}%
  {\topsep}   
  {\topsep}   
  {\rmfamily}  
  {}          
  {\itshape} 
  {.}         
  { }         
  {}          
\theoremstyle{propositional}
\newtheorem{thm}{Theorem}[section]
\newtheorem{prop}[thm]{Proposition}
\newtheorem{lem}[thm]{Lemma}

\theoremstyle{descriptive}
\newtheorem{deff}[thm]{Definition}

\newtheorem{remark}[thm]{Remark}

\makeatletter
\renewenvironment{proof}[1][\proofname]{\par
  \pushQED{\qed}%
  \normalfont
  \trivlist
  \item[\hskip\labelsep
        \itshape
    #1\@addpunct{.}]\ignorespaces
}{%
  \popQED\endtrivlist\@endpefalse
}
\makeatother
\newdimen\extrawidth
\def\iintlim#1#2{\setbox0\hbox{$\scriptstyle#1$}%
        \setbox1\hbox{$\scriptstyle#2$}%
        \extrawidth=\wd1 \advance\extrawidth-\wd0
        \ifdim\extrawidth<0pt \extrawidth=0pt\fi%
        \int_{#1\kern\extrawidth \kern .5em}^{#2\kern -\wd1} \kern -.5em%
}
\renewcommand{\emptyset}{\varnothing}
\def\vint{\mathop{\mathchoice%
          {\setbox0\hbox{$\displaystyle\intop$}\kern 0.22\wd0%
           \vcenter{\hrule width 0.6\wd0}\kern -0.82\wd0}%
          {\setbox0\hbox{$\textstyle\intop$}\kern 0.2\wd0%
           \vcenter{\hrule width 0.6\wd0}\kern -0.8\wd0}%
          {\setbox0\hbox{$\scriptstyle\intop$}\kern 0.2\wd0%
           \vcenter{\hrule width 0.6\wd0}\kern -0.8\wd0}%
          {\setbox0\hbox{$\scriptscriptstyle\intop$}\kern 0.2\wd0%
           \vcenter{\hrule width 0.6\wd0}\kern -0.8\wd0}}%
          \mathopen{}\int}
\def\vintslides{\mathop{\mathchoice%
          {\setbox0\hbox{$\displaystyle\intop$}\kern 0.22\wd0%
           \vcenter{\hrule height 0.04em width 0.6\wd0}\kern -0.82\wd0}%
          {\setbox0\hbox{$\textstyle\intop$}\kern 0.2\wd0%
           \vcenter{\hrule height 0.04em width 0.6\wd0}\kern -0.8\wd0}%
          {\setbox0\hbox{$\scriptstyle\intop$}\kern 0.2\wd0%
           \vcenter{\hrule height 0.04em width 0.6\wd0}\kern -0.8\wd0}%
          {\setbox0\hbox{$\scriptscriptstyle\intop$}\kern 0.2\wd0%
           \vcenter{\hrule height 0.04em width 0.6\wd0}\kern -0.8\wd0}}%
          \mathopen{}\int}
\def\Xint#1{\mathchoice
{\XXint\displaystyle\textstyle{#1}}%
{\XXint\textstyle\scriptstyle{#1}}%
{\XXint\scriptstyle\scriptscriptstyle{#1}}%
{\XXint\scriptscriptstyle\scriptscriptstyle{#1}}%
\!\int}
\def\XXint#1#2#3{{\setbox0=\hbox{$#1{#2#3}{\int}$ }
\vcenter{\hbox{$#2#3$ }}\kern-.6\wd0}}

\def\dashint{\Xint-}
\def\avint{\Xint-}

\def\XXsum#1#2#3{{\setbox0=\hbox{$#1{#2#3}{\sum}$ }
\vcenter{\hbox{$#2#3$ }}\kern-.5\wd0}}

\DeclareMathOperator{\grad}{grad}

\DeclareMathOperator{\Lip}{Lip}

\newcommand{\tf}{\tilde{f}}
{\catcode`p =12 \catcode`t =12 \gdef\eeaa#1pt{#1}}      
\def\accentadjtext#1{\setbox0\hbox{$#1$}\kern   
                \expandafter\eeaa\the\fontdimen1\textfont1 \ht0 }
\def\accentadjscript#1{\setbox0\hbox{$#1$}\kern 
                \expandafter\eeaa\the\fontdimen1\scriptfont1 \ht0 }
\def\accentadjscriptscript#1{\setbox0\hbox{$#1$}\kern   
                \expandafter\eeaa\the\fontdimen1\scriptscriptfont1 \ht0 }
\def\accentadjtextback#1{\setbox0\hbox{$#1$}\kern       
                -\expandafter\eeaa\the\fontdimen1\textfont1 \ht0 }
\def\accentadjscriptback#1{\setbox0\hbox{$#1$}\kern     
                -\expandafter\eeaa\the\fontdimen1\scriptfont1 \ht0 }
\def\accentadjscriptscriptback#1{\setbox0\hbox{$#1$}\kern 
                -\expandafter\eeaa\the\fontdimen1\scriptscriptfont1 \ht0 }

\renewcommand{\phi}{\varphi}

\newcommand{\R}{\mathbb{R}}

\newcommand{\N}{\mathbb{N}}

\newcommand{\Z}{\mathbb{Z}}

\newcommand{\limminus}{{\mathchoice{\raise.17ex\hbox{$\scriptstyle -$}}
                {\raise.17ex\hbox{$\scriptstyle -$}}
                {\raise.1ex\hbox{$\scriptscriptstyle -$}}
                {\scriptscriptstyle -}}}
\newcommand{\limplus}{{\mathchoice{\raise.17ex\hbox{$\scriptstyle +$}}
                {\raise.17ex\hbox{$\scriptstyle +$}}
                {\raise.1ex\hbox{$\scriptscriptstyle +$}}
                {\scriptscriptstyle +}}}
\makeatletter
\newcommand{\setcurrentlabel}[1]{\def\@currentlabel{#1}}
\makeatother

\makeatletter

\newcommand{\Rmnum}[1]{\expandafter\@slowromancap\romannumeral #1@}
\makeatother

\numberwithin{equation}{section}

\title[Discrete Approximations of Metric Spaces]{Discrete Approximations of Metric Measure Spaces of Controlled Geometry}
\author{James T. Gill}
\thanks{Acknowledgements: We thank IPAM and NSF for support.  The first author thanks the University of Cincinnati for their hospitality.  The second author wishes to thank Saint Louis University for their hospitality.  We also wish to thank Nageswari Shanmugalingam for her support and inspiration.
}

\author{Marcos Lopez}
\thanks{The first author was partially supported by the NSF grant DMS-1004721 and the second author was partially supported by the NSF grant DMS-1200915.}
\begin{document}

\maketitle
\begin{center}
\vspace{-.5cm}
{\small\hspace{.4cm}\email{jgill5@slu.edu} \hspace{0.4cm} \email{lopezms@mail.uc.edu*}}
\end{center}
\begin{abstract}
{ We find a necessary and sufficient condition for a doubling metric space to carry a $(1,p)$-Poincar\'e inequality.  The condition involves discretizations of the metric space and Poincar\'e inequalities on graphs.}
\end{abstract}
 \\ \\

\section{Introduction}
It is well known that a doubling metric space which also supports some type of Poincar\'e inequality enjoys many other useful properties (see \cite{H} and the upcoming \cite{HKST} for examples).  However, it is often the case that when presented with an arbitrary metric measure space $(X, d_X, \mu)$, verifying that it satisfies a Poincar\'{e}-type inequality is difficult. In this paper we present a method of discretizing a metric measure space that is doubling and supports a Poincar\'{e} type inequality (see Section 2 for these definitions).  The constructed discretized space will (1) retain the doubling property as well as (2) support its own Poincar\'{e} type inequality. We also show that a doubling metric measure space $(X,d_X,\mu)$ only has a Poincar\'e inequality if some such discretization exists. By discretizing the space, the advantage is that we may verify these properties by checking only a finite number of points for each ball $B \subset X$. With enough symmetry or regularity of a space, this may be simple, {as we see in the example in Section \ref{sec:examples}}. This transforms the possibly difficult problem of verifying the doubling and Poincar\'{e} properties into a problem that is more computationally feasible. The method of discretization has been well studied in the study of analysis on metric measure spaces. L. Ambrosio, M Colombo, and S. Di Marino used an analog of dyadic cubes, introduced by M. Christ, to study the theory of Sobolev spaces on metric measure spaces (see \cite{Am} and \cite{M}). This approach to studying metric measure spaces follows the work of R.R. Coifman and G. Weiss (see \cite {CW}). We will use a method of using maximally $\epsilon$-separated subsets to discretize metric measure spaces that also follows their work, but requires different assumptions on our space. 

Our method is analogous to that used by P. Herman, R. Peirone, and R. Strichartz to study $p$-energy on the Sierpinski gasket (See \cite{HPS}). Their work focused on constructing energy forms on the gasket via natural energy forms on discrete approximating graphs. In our paper we are not interested in approximating energies on the metric space since the metric space is already equipped with the energy from from the upper gradient structure; we focus instead on Poincar\'e inequalities (which, in turn, are not available in \cite{HPS}). There has been work in modifying the natural metric on the Sierpinski gasket in order to ensure a Poincar\'e inequality, and this {\em harmonic Sierpinski gasket} discussed in \cite{Ka} by N. Kajino,  and Kusuoka in \cite{Ku}, uses a metric change that may not be bi-Lipschitz, and does not preserve many aspects of the original space. There has also been work studying the limits of Dirichlet forms on post-critically finite fractals following the work of Barlow and Bass (see \cite{BaBa}, \cite{HK}, \cite{KZ}, and \cite{KS}). However, these works are done on connected metric graphs, where here we present a discrete condition on highly non-connected spaces.  Also, recent notes by J. Cheeger and B. Kleiner \cite{CK} and \cite{CKa} studies Poincar\'e inequalities on discrete spaces and inverse limits, a different approach than our note here that restrict their scope to metric spaces that are topologically of dimension 1. In their work, they show that a metric space satisfies the $(1,1)$-Poincar\'e inequality if it is possible to construct an ``inverse limit", or equivalently a Gromov-Hausdorff limit. In our paper, we show an approach that holds for $(1,p)$-Poincar\'e inequalities with $p \geq 1$.


The setting considered in this paper is that of a general metric space $X$, endowed with a metric $d_X$ and a doubling Borel regular measure $\mu$; see Section 2 for precise definitions. We will construct a metric measured graph $(V, d_V, m)$  based on $X$ such that $m$ is a doubling measure, and show that $V$ also supports a Poincar\'{e} type inequality when $X$ does. {Throughout}  this note $1 \leq p < \infty$.  Our main results are as follows:

\begin{thm} \label{XtoV}
Let $(X,d_X,\mu)$ be a complete doubling metric measure space that supports a $(1,p)$-Poincar\'e inequality.  Then any discretized space $(V,d_V,m)$, constructed from $(X, d_X, \mu)$ in the manner given in Section \ref{sec:constr}, is also doubling and supports a $(1,p)$-Poincar\'e inequality with data quantitatively derived from the data of $(X, d_X,\mu)$.
\end{thm}

We then turn our attention to the converse of Theorem \ref{XtoV}, which requires some preliminary definitions.  Let $(V, d_V, m)$ be a graph with metric $d$ and measure $m$.  By {\em graph} we mean a set of vertices $V$ with an associated edge set $E$, which we suppress in the notation by only refering to the graph as $V$.  In this paper for each graph there is a constant $\epsilon_V$ so that if vertices $x$ and $y$ are connected by an edge, the distance $d_V(x,y) = \epsilon_V$.  Distances between other vertices $x$ and $y$ are defined via $n\cdot \epsilon_V$ where $n$ is the smallest length of a sequence $x = x_0, x_1, x_2, \ldots, x_n=y$ where $x_i$ and $x_{i+1}$ are connected by an edge.  By $B_V (x,r)$ we mean all vertices of distance strictly less than $r$ from $x$.  The measure $m$ is simply an assignment of a positive mass to each vertex and, as it is discrete, it is defined on all subsets of $V$.  For a (discrete) graph $(V,d_V,m)$ and a metric measure space $(X,d, \mu)$ an {\em embedding} of $V$ into $X$ is a one-to-one map from the vertex set $V$ into the space $X$.  For a sequence of graphs $(V_i, d_{V_i}, m_i)$ with $V_i \subset V_{i+1}$ for $i \geq 0$ by a {\em nested embedding} of $(V_i, d_{V_i}, m_i)$ into $X$ we mean a sequence of embeddings $n_i: V_i \to X$ such that $n_{i+j} |_{V_i} = n_i$ for $i, j \geq 0$.  Note this definition of nested embedding does not imply that if $x$ and $y$ are connected by an edge in $V_i$, they are still connected by an edge in $V_{i+1}$.  In general, they will not be connected by an edge in $V_{i+1}$.
{Our second result shows that the discretezation from Theorem \ref{XtoV} can also yield information about the space $(X,d,\mu)$.}
\begin{thm}\label{VtoX}
Let $(X,d,\mu)$ be a complete doubling metric measure space.  Then $(X,d,\mu)$ supports a $(1,p)$-Poincar\'e inequality if and only if there exists a nested embedded sequence of graphs $(V_i, d_{V_i}, m_i)$ into $X$ such that
\begin{enumerate}
\item[(1)] The Hausdorff distance, $d_H (n_i(V_i), X) = H_i$, is finite for all $i \geq 0$ and $H_i \to 0$ as $i \to \infty$.
\item[(2)] There is a uniform $L > 1$ such that for all $i \geq 0$ and all $x,y \in V_i$,
\[ \frac{1}{L} d (n_i(x), n_i(y)) \leq d_{V_i} (x,y) \leq L \, d (n_i(x), n_i(y)) \]
\item[(3)] There is a uniform $K > 1$ such that for all $i \geq 0$ all $r > H_i$ and $x \in V_i$
\[ \frac1K \leq \frac{m_i (B_{V_i}(x,r))}{\mu (B_X (x,r))} \leq K\]
\item[(4)] $(V_i, d_{V_i}, m_i)$ are all doubling metric measure spaces with uniform doubling constant.
\item[(5)] $(V_i, d_{V_i}, m_i)$ all support a (1,p)-Poicar\'e inequality with uniform {data}.
\end{enumerate}
\end{thm}

For the definition of Hausdorff distance see (\ref{Hdist}), for the definition of doubling metric space see Section 2, and for the type of Poincar\'e inequality assumed see (\ref{dpi}).

In Section 2, we review applicable definitions for this paper. Section~3 focuses on constructing $V$, a discretization of $X$, and endowing the set with a metric $d_V$ and measure $m$ that are derived from $d_X$ and $\mu$.  In Section 4 we verify that $(V, d_V, m)$ also satisfies the doubling property. Section 5 is dedicated to showing that $(V, d_V, m)$ satisfies a discretized version of the Poincar\'{e} inequality. Sections 4 and 5 together provide the proof of Theorem~\ref{XtoV}. In Section~6, we review the pertinent definitions of pointed measured Gromov-Hausdorff convergence {which will be necessary to our proof of Theorem \ref{VtoX}.} Section~7 discusses an example from Euclidean space showing the necessity of the conditions in Theorem \ref{VtoX} {as well as showing how one might check for a discrete Poincar\'e inequality}. Finally, Section 8 is dedicated to {the proof of} Theorem \ref{VtoX}.

\section{Preliminaries}
In this section we introduce some necessary definitions. All of this section is standard and may be skipped by the expert on metric measure spaces.  A nontrivial locally finite Borel regular measure $\mu$ on a metric space $(X,d_X)$ is called a \textit{doubling measure} if every metric ball, $B$, has positive and finite measure and there exits a constant, $C\geq 1$, such that
\[
\mu(B(x, 2r)) \leq C\mu(B(x,r))
\]
for each $x$ in $X$ and $r > 0$. We call the triple $(X,d_X,\mu)$ a \textit{doubling metric measure space} if $\mu$ is a doubling measure on $X$. The smallest constant $C \geq 1$ such that the above inequality holds is referred to as the \textit{doubling constant $C_\mu$ of} $\mu$. An $\epsilon$-\textit{separated set}, $\epsilon > 0$, in a metric space is a set such that every two distinct points in the set are  at least $\epsilon$ distance apart. Given a metric space $X$, an $\epsilon$-separated set $A \subset X$ is said to be $maximal$ if for any $x \in X\backslash A$, the distance from $x$ to $A$ is less than $\epsilon$. The metric $d_X$ is said to be \textit{doubling metric with constant N} if $N \geq 1$ is an integer such that for each ball $B(x,r) \subset X$, every $\frac{r}{2}$ - separated set in $B(x,r)$ has at most $N$ points.

It is easy to show that if $(X, d_X,\mu)$ is a doubling metric measure space, then $d_X$ is also a doubling metric with some constant that depends only on $C_\mu$: let $B(x,r)$ be given. Let $A$ be some maximal $r/2$ separated set of $X$. To see that a maximal $r/2$-separated subset of $X$ exists, see Chapter 10 of \cite{H}. If $A \cap B(x,r)$ contains $I$ points,  $a_1, a_2, \ldots, a_I$, where $I \subset \N$ is an indexing set, then the set of balls $\{B(a_i, r/2)\}_{i \in I}$ cover $B(x,r)$ by the maximality of $A$.   Note that the balls $\{B(a_i, r/4)\}_{i \in I} \subset B(x,2r)$, and are pairwise disjoint. Then for $N \in I$,
 \[
 N\mu(B(x,2r)) \leq \sum_{i=1}^N C_\mu^3\mu(B(a_i,\frac{r}{2})) \leq C_\mu^4\sum_{i=1}^N\mu(B(a_i,\frac{r}{4})) \leq C_\mu^4\mu(B(x,2r)).
 \]

Thus, $N \leq C_\mu^4$ when $\mu$ is locally finite. Notice that the assumption that $\mu$ is positive on balls also implies that $N$ must be a finite number, because to be infinite would imply that $\mu(B(x,2r))$ is infinite which contradicts our assumption of $\mu$ being locally finite. Since this holds for all $N \in I$, then the cardinality of $I$ must also be less than or equal to $C_\mu^4$.

The above metric property was {formulated} by Coiffman and Weiss in \cite{CW}, and  proves to be of great importance in the study of Sobolev spaces. In particular, doubling spaces can be shown to be separable. If, in addition to being doubling, a metric space is complete, then it is proper.  Note that as in the statements of Theorems \ref{XtoV} and \ref{VtoX} will always assume that $X$ is equipped with a complete metric $d$, and a locally finite Borel regular measure $\mu$.

Let $u$ be a real-valued measurable function on $X$. A non-negative Borel function  $\rho: X \rightarrow [0,\infty]$ is said to be an \textit{upper gradient} of $u$ if for all compact rectifiable paths $\gamma: [a,b] \to X$, the following inequality holds:
\[
|u(\gamma(a)) - u(\gamma(b))| \leq \int_\gamma \rho \,ds
\]
where $\,ds$ is the arc-length measure on $\gamma$, induced by the metric $d_X$ on $X$. (see Chapter 7 of \cite{H}).
A separable metric measure space $(X, d_X,\mu)$ is said to \textit{support a $(1,p)$-Poincar\'{e} inequality} if every ball $B \subset X$  has positive and finite measure and if there exist constants $C > 0, \lambda \geq 1$ such that
\begin{align}
\dashint_B|u - u_B|d\mu \leq Cr\left(\avint_{\lambda B} \rho^p\,d\mu\right)^{\frac{1}{p}} \label{pi}
\end{align}
for every measurable function $u: X \rightarrow \mathbb{R}$ that is integrable on balls and every upper gradient $\rho$ of $u$. In the above inequality, when the center and radius are clear from context, $B$ is written as the shorthand of $B(x,r)$ and $\lambda B:= B(x,\lambda r)$. The notation of $\dashint_B$ is the average integral over the ball $B$. That is,
  \[
  \dashint_B u\, d\mu := \frac{1}{\mu(B)}\int_B u \, d\mu \ \ =:\ u_B
  \]
  for any integrable function $u$ on $B$.  The parameters $p, C,$ and $\lambda$ are called the \textit{data} of the Poincar\'{e} inequality.

\section{Construction of the approximating graphs}
\label{sec:constr}

Let $(X,  d_X, \mu)$ be a doubling metric measure space with doubling constant $C_\mu$, and let $A \subset X$ be a maximal $\epsilon$-separated set for some given $\epsilon >0.$ For each $x \in A$, we associate a vertex $\tilde{x} \in V_\epsilon =: {V}.$
We say that $\tilde{x} \sim \tilde{y}$ if and only if $\epsilon \leq \,d(x,y)\leq 3\epsilon$. We let $\sim $ define an edge set.  We will use this relation between points in $V$ to define a metric on $V$. It is worth noting here, that $(V, \sim)$ is a discrete graph, possessing only a discrete topology. In Section 8, this graph will be extended to a connected graph, but for the majority of this paper, all of the calculations involving $V$ will only use points on the vertex set $V$. To highlight this point, we will often only refer to the vertex set $V$ which is in a 1-1 correspondence with the set $A\subset X$, and therefore may be thought of as an embedding in $X$. This canonical embedding can be given by identifying $x\in A$ with $\tilde{x} \in V$. Notice that we require the distance between two points of $A$ to be positive in order for a corresponding edge to be made in $(V,\sim)$. This is to ensure that $(V,\sim)$ has no loops of zero length. We define a distance on $V$, denoted $d_V$, such that $d_V(\tilde{x}, \tilde{y})= \epsilon$ for all $\tilde{x}\sim \tilde{y}$. We extend this distance function for $\tilde{x}$ and $\tilde{y}$ that do not share an edge the distance between them is the obvious one as stated in the introduction after the statement of Theorem \ref{XtoV}.  It is clear that $d_V$ is a metric. The ball centered at $\tilde{x}$ with radius $r$ is denoted $B_V(\tilde{x}, r):= \{\tilde{y}\in V | d_V(\tilde{x}, \tilde{y}) < r\}$. Note that when $r \leq \epsilon$, we have that $B(\tilde{x}, r) = \{\tilde{x}\}$.  We use $B_V$ to denote balls in the graph metric and $B_X$ to denote balls in the original metric space.

It can be shown that a complete and doubling metric measure space that supports a Poincar\'e inequality is {\it $L$-quasiconvex}. That is, for the space $X$, there is a constant $L \geq 1$, depending only on the doubling constant and the data from the Poincar\'{e} inequality, such that each pair of points $x$, $y \in X$ can be joined by a rectifiable curve $\alpha$ in $X$ such that length$(\alpha) \leq L d_X(x, y)$. This is result is due to S. Semmes, but proofs may be found in \cite{K} and \cite{KH}. The quasiconvexity of the metric measure spaces allow us to observe a very useful property addressed in the following proposition.

\begin{prop}
\label{blipcomp}
 For a complete doubling metric space $(X,d_X)$ that is quasiconvex with constant $L$, the canonical embedding of $(V, d_V)$ into $X$ is bi-Lipschitz:
\[ \frac{1}{L+1} d_V(\tilde{x},\tilde{y}) \leq d_X(x,y) \leq 3d_V(\tilde{x},\tilde{y}) \]
\end{prop}

\begin{proof}
We begin by showing the first inequality. Let $\tilde{x}$ and $\tilde{y}$ be two points in the vertex set $V$, and let $x$ and $y$ be their respective corresponding points in $X$.
If $d_X(x,y) \leq 3\epsilon$ then $\tilde{x} \sim \tilde{y}$ or $\tilde{x} = \tilde{y}$. Thus, either $d_V(\tilde{x}, \tilde{y}) = \epsilon$ or $d_V(\tilde{x}, \tilde{y}) = 0$. The latter case satisfies the proposition trivially, and the former case follows by seeing that
\[
\frac{\epsilon}{L+1} \leq 3\epsilon.
\]
Thus, without loss of generality, we assume that $d_X(x,y) > 3\epsilon$.  Let $\gamma$ be a rectifiable curve from $x$ to $y$ such that length$(\gamma) \leq L d_X(x,y)$. Let $T:= $length$(\gamma)$, and notice that $T > 3\epsilon$. Since $\gamma$ is rectifiable, we assign it the arc-length parameterization. Choose $K$ as the smallest integer such that $T \leq K\epsilon < Ld_X(x,y)+\epsilon$. Notice that by assumption, $d_X(x,y) > 3\epsilon$, so such $K$ exists. For $i = 0, 1, \dots, K -1$ we choose $t_i = i\epsilon$, and define $t_K := T$. Then, for $i = 1, \dots, K$ there are subcurves $\gamma_i := \gamma([t_{i-1}, t_i])$ of $\gamma$ such that length$(\gamma_i) = \epsilon$ with the exception of $\gamma_{K}$ which may have length less than or equal to $\epsilon$. Let $x_i:= \gamma(t_i)$. By the maximality of $A$, for each $x_i$ there exists a point $z_i \in A$ such that $d_X(x_i, z_i) \leq \epsilon$. It is clear that we can choose $z_0 = x$ and $z_K = y$. We find that for $i = 1, \dots, K$,
\[
d_X(z_{i-1}, z_i) \leq d_X(z_{i-1}, x_{i-1}) +d_X(x_{i-1}, x_i) + d_X(x_i, z_i) \leq 3\epsilon.
\]
Since each $z_i$ is in $A$, then it has a corresponding point $\tilde{z}_i$ in $V$. Hence, for each $i$ we have that $\tilde{z}_i \sim \tilde{z}_{i-1}$  or $z_i = z_{i-1}$. So $d_V(z_i, z_{i-1})\leq \epsilon$, and
\begin{align*}
d_V(\tilde{x}, \tilde{y})\leq K\epsilon \leq Ld_X(x,y) + \epsilon \leq Ld_X(x,y) + d_X(x,y) & = (L+1)d_X(x,y).
\end{align*}

The second inequality follows easily from the definition of the distance on $V$ and the triangle inequality on $X$.
\end{proof}

We now wish to equip $V$ with a measure $m$  which is related to $\mu$. By the maximality of $A$ (and its one to one correspondence to $V$), $X=\underset{\tilde{x} \in V}{\bigcup} B_{X}(x, \epsilon)$. For any $W\subset V$, we define
\[
m(W) := \underset{\tilde{y} \in W}{\sum}\mu(B_X(y,\epsilon)).
\]
For example, if $r < \epsilon$, then $m(B_V(\tilde{x}, r)) = \mu(B_X(x,\epsilon))$. In particular, for any $\tilde{x} \in V$, we set $m(\tilde{x}) = \mu(B_X(x,\epsilon))$.
We see that $m$ is a measure on the $\sigma$ - algebra generated by the open balls in $V$. We note that, in general,
\[
m(W) \neq \mu\left(\bigcup_{\tilde{y} \in W} B_X(y,\epsilon)\right).
\]
Hence, $(V, d_V, m)$ is a metric measure space.
For $W \subset V$ and $u: V \rightarrow \R$,  the definition of $\int_W u(\tilde{x}) dm(\tilde{x})$ is given by
\[
\int_W u(\tilde{x})\,dm(\tilde{x}) := \underset{\tilde{x} \in W}{\sum} u(\tilde{x})m(\tilde{x}).
\]
When the context is clear, we will often use the following notation for a fixed $\tilde{x} \in V$:
\[
\int_{\tilde{x} \sim \tilde{y}} u : = \sum_{\tilde{x} \sim \tilde{y}} u(\tilde{y})
\]
This helps us  when we wish to sum only over neighbors, but should not be confused with an integral over the function $u$, as this is in itself a function on $V$ evaluated at the point $\tilde{x}$. This should be made clear from the lack of associated measure in the notation. In a similar manner to (\ref{pi}), we may define $u_{B_V}$ and $\avint_{B_V}$ for a function $u: V\rightarrow \R$:
\[
u_{B_V} := \avint_{B_V} u(\tilde{x})dm(\tilde{x}) := \frac{1}{m(B_V)} \int_{B_V} u(\tilde{x})dm(\tilde{x})
\]
to echo the meaning of $\dashint_{B_X} f(x)dx$. That is,  $\avint_{B_V} u(\tilde{x})dm(\tilde{x})$ is an  $m$-weighted average value of $u$ over the ball $B_V$.
	
We now describe a \textit{discretized} version of the $(1,p)$-Poincar\'e inequality that was introduced by I. Holopainen and P. Soardi in \cite{HS}.
\begin{deff}
We say that $V$ \textit{supports a (discrete) $(1,p)$-Poincar\'{e} inequality} if there exist some constants $C >0$ and $\lambda \geq 1$ such that for all functions $u: V \rightarrow \mathbb{R}$, and each $B_V= B(\tilde{v},r) \subset V$,
\begin{align}
\avint_{B_V}|u(\tilde{x}) - u_{B_{V}}| dm(\tilde{x}) \leq C r \left(\avint_{\lambda B_V} \left( \int_{\tilde{x} \sim \tilde{y}} \frac{|u(\tilde{x})-u(\tilde{y})|^p}{\epsilon^p} \right) dm(\tilde{x})\right)^\frac{1}{p}. \label{dpi}
\end{align}
\end{deff}

Note that the quantity
\[
\frac{|u(\tilde{x})-u(\tilde{y})|}{\epsilon}
\]
can be seen as a type of upper gradient when compared to $\rho$ in (\ref{pi}) if we consider the edge from $\tilde{x}$ to $\tilde{y}$ as an isometric copy of the interval $[0, \epsilon]$.
We wish to rename the quantity on the right hand side of (\ref{dpi}) for simplicity of exposition. Given a function  $u$ on the vertex set $V$, and $\tilde{a} \in V$ we define
\[
\left|\grad (u(\tilde{a}))\right| := \int_{\tilde{b} \sim \tilde{a}} \frac{| u(\tilde{b})-u(\tilde{a}) |}{\epsilon}.
\]
This $|\grad(u)|$ function is often referred to as the ``$p$-Laplacian", and is used to define the ``graph energy" on $V$ (see \cite{HPS} for the case that $V$ is a discretization of the Sierpinksi gasket). Note that although $\epsilon$ is a fixed number here, later in this note we will be considering a sequence of graphs constructed with different $\epsilon$ values. We suppress the dependency of $|\grad(u)|$ on $\epsilon$ in the notation.
Then (\ref{dpi}) becomes
\begin{equation}
\label{newdpi}
\avint_{\tilde{x} \in B_V}|u(\tilde{x}) - u_B| dm(\tilde{x}) \leq C' r \left(\avint_{\lambda B_V} \left|\grad u(\tilde{x})\right|^p  dm(\tilde{x})\right)^\frac{1}{p}
\end{equation} where $C'$ depends on $C, p,$ and the maximal degree of the graph, as we describe now.
Note that the metric doubling constant of $N \leq C_\mu^4$ implies that the maximal degree of our graph is bounded by $C_\mu^4$. That is, any $\tilde{x} \in V$ has at most $C_\mu^4$ neighboring points.
 Thus, we see that
 \[
 \left|\grad(u(\tilde{a}))\right|^p \approx \int_{\tilde{a} \sim \tilde{b}} \left(\frac{|u(\tilde{b}) - u(\tilde{a})|}{\epsilon} \right)^p
 \]
 with $\approx$ meaning that the two differ by a bounded multiplicative constant. Because of this fact, we may substitute the righthand side of (\ref{dpi}) with the right hand side of (\ref{newdpi}) and absorb this multiplicative constant into the constant from (\ref{dpi}). This final version (\ref{newdpi}) of the discrete Poincar\'e inequality is the one we use. 
 
 We make a note of a difference between the traditional Poincar\'e inequality (\ref{pi}) and this discretized version. If a space supports some $(1,p)$-Poincar\'e inequality for any $p\geq 1$, then a simple topological consequence is that the space is connected. Since we will be working with discrete spaces, it is clear that we may not use the traditional Poincar\'e inequality (\ref{pi}) for these. The discrete graphs in this paper are made of isolated points with positive distance between them. That is to say, these spaces are highly disconnected. However, there are still some properties that we can obtain from our discrete version of the Poincar\'e inequality that are in line with the traditional version. For example, if $u$ is a function on a space, and $u$ has the constant function $0$ as an upper gradient, we would like to conclude that $u$ is a constant function. This is consequence of the space supporting a Poincar\'e inequality in the traditional sense, but it is also a consequence of a discrete space supporting a discrete Poincar\'e inequality.
  

\section{Doubling property of the approximating graph}
Since the underlying space $(X, d_X, \mu)$ is doubling, it is natural to question whether or not the constructed graph shares this property. The aim of this section is to prove the following lemma:
\begin{lem}\label{DoubXV}
Suppose that $(X, d_X, \mu)$ is doubling with constant $C_\mu$, and $X$ is quasiconvex with constant $L$. 
Let A be a maximal $\epsilon$-separated subset of $X$, and let $(V,d_V, m)$ be constructed from $(X,d_X,\mu)$ as before. Then $m$ is a doubling measure on $V$.
\end{lem}
The space $(V,d_V, m)$ being a doubling space allows us the use of many results of harmonic analysis that extend to doubling spaces. For example, the Lebesgue Differentiation Theorem extends to metric measure spaces under the doubling property. This result is due to the reliance of the Hardy-Littlewood maximal inequality on the doubling property. There are many other extensions that come from having a doubling property. Coifman and Weiss (see \cite{CW}) in particular were pioneers of verifying properties of doubling spaces as related to harmonic analysis. J. Luukkainen and E. Saksman showed that every complete doubling metric space carries a doubling measure (see \cite{LS}). It is worth noting that the assumption of completeness is essential here. For example, Saksman showed that every metric space without isolated points has a dense subset that does not carry a doubling measure (see \cite{S}).  In Section 6 of this paper we will see that having the doubling property is vital to verifying pointed-measured Gromov-Hausdorff convergence  of sequences of discretized metric measure spaces.

Before we begin the proof, we note that due to Proposition \ref{blipcomp},
\begin{equation} A \cap B_X\left(x, \frac{r}{L+1}\right) \subset  B_V(\tilde{x}, r) \subset B_X(x, 3r).  \label{ballLipeq} \end{equation}

\begin{proof}[Proof of Lemma 4.1]
 We must show that there is some constant $C_m \geq 1$ such that for any $\tilde{x}\in V$, and $r>0$, $m(B_V(\tilde{x}, 2r)) \leq C_m\,m(B_V(\tilde{x}, r))$.
Fix $\tilde{x} \in V$. The case where $0< r< \epsilon$ is easily seen due to the uniform bound on the degree of the graph. That is, in this particular case, we have that $m(B_V(\tilde{x}, r)) = m(\{\tilde{x}\}) = \mu(B_X(x, \epsilon))$, and $m(B_V(\tilde{x}, 2r)) \leq m(B_V(\tilde{x}, 2\epsilon))$. We see that
\begin{align*}
m({B_V}(\tilde{x},2\epsilon)) &\leq \sum_{\tilde{y} \sim \tilde{x}} \mu({B_X}(y,\epsilon))\\
& \leq \text{deg}(\tilde{x})\mu(B_X(x,2\epsilon))\\
& \leq C_\mu^4\,\mu(B_X(x,2\epsilon))\\
& \leq C_\mu^5\,\mu(B_X(x,\epsilon)) \\
& = C_\mu^5\,m(B_V(\tilde{x},\epsilon)).
\end{align*}
Thus, we will consider the case where $\epsilon \leq r$. By definition, \[m(B_V(\tilde{x},2r)) = \underset{\tilde{y} \in 2B_V}{\sum} \mu(B_X(y, \epsilon)).\] This may be a problem if the sum is infinite. However, since $X$ is a doubling metric measure space, there exists an $N < \infty$ such that there are at most $N$ points of A, $y_1, y_2, \dots, y_N$, in each ball $B_X(y, 2\epsilon)$. Recall that we verified that $N\leq C_\mu^4$ in Section 2. This fact, along with the assumption that $\mu$ is locally finite bypasses such a problem, and the sum will be finite.   For this calculation set $\alpha = \lceil \log_2 (L+1) \rceil$. For the first inequality below, we use the fact that $B_V(\tilde{x}, 2r) \subset B_X(x, 3(2r + \epsilon))$ by Proposition (\ref{blipcomp}), and the appearance of $C_\mu^4$ is due to the overlapping constant discussed in Section 2.
\begin{align*}
m(B_V(\tilde{x},2r)) = \sum_{\tilde{y} \in 2B_V} \mu(B_X(y,\epsilon)) & \leq C_\mu^4 \mu(B_X(x, 3(2r+\epsilon)))\\
& \leq  C_\mu^{6+2\alpha} \mu\left(B_X\left(x, \frac{2r+\epsilon}{L+1}\right)\right) \\
& \leq C_\mu^{8+2\alpha} \mu\left(B_X \left( x, \frac{r/2 + \epsilon/4}{L+1}\right)\right)  \\
& \leq C_\mu^{8+2\alpha} \mu\left(B_X \left( x, \frac{r}{L+1} \right) \right) \\
& \leq C_\mu^{8+2\alpha} \, \sum_{\tilde{y} \in B_V} \, \mu(B_X(y,  \epsilon)) \\
& = C_\mu^{8+2\alpha} \, m(B_V(\tilde{x}, r))
\end{align*}

Hence, $m$ is a doubling measure on $V$ with doubling constant $C_m = C_\mu^{8+2\alpha}$.
\end{proof}

We also note that the above proof can be easily modified to show the following:
\begin{lem}
\label{mcompmu}
Suppose $(X, d_X,\mu)$ is doubling with constant $C_\mu$, and is quasiconvex with constant $L$. Let $A$ be a maximal $\epsilon$-seperated subset of $X$ and let $(V, d_V,m)$ be constructed as above.  Then for $x \in A$ and $r \geq \epsilon$, there exists a constant $K$ such that
\[ \frac{1}{K} m(B_V(\tilde{x},r)) \leq \mu(B_X(x,r)) \leq K m(B_V(\tilde{x}, r)) \]
where $K$ depends only on $C_\mu$ and $L$.
\end{lem}

\begin{remark} \label{drmk} In other words, the lemma says that $m$ and $\mu$ are comparable at scales larger than $\epsilon$.
Also, the result holds with a multiple $a>1$ of $r$ though repeated use of the doubling property.  Hence $m(B_V(\tilde{x},r))$ and $\mu(B_X(x,a r))$ are comparable with the constant now depending also on $a$.
\end{remark}

\section{Proof of Theorem \ref{XtoV}}
\begin{proof}
In this section we will show that $V$ supports a Poincar\'{e} inequality in the sense of (\ref{newdpi}). We will do this essentially by transforming a given function $\tilde{f}:V\rightarrow\mathbb{R}$ into a function $f:X\rightarrow \mathbb{R}$ by employing a partition of unity, and using the fact that $X$ supports a Poincar\'{e} inequality in the sense of (\ref{pi}) and then reinterpreting this inequality back to the discrete function $f$. The rest of the proof lies only in checking the details of this sketch.  Let $A$ be the maximally $\epsilon$-separated subset of $X$ that is associated with $V$ (as in Section 2). Fix $a \in A$, and let $\psi_a: X \rightarrow \mathbb{R}$ be given by
\[
\psi_a(x):= \min \left\{ 1, \frac{d_{ X}(x,X\backslash B_{X}(a,2\epsilon))}{\epsilon}\right\}.
 \]
Notice that if $x \in B_{ X}(a, \epsilon)$, then $\psi_a(x) = 1$, and if $x \notin B_{ X}(a, 2\epsilon)$, then $\psi_a(x) = 0$. Let $\phi_a: X\rightarrow \mathbb{R}$ be defined as follows:
\[
\varphi_a(x) := \frac{\psi_a(x)}{\underset{b\in A}{\sum} \psi_b(x)}.
\]
For any $a \in A, \phi_a$ is a Lipschitz function with Lipschitz constant equal to $\frac{C}{\epsilon}$, where $C$ only depends upon the doubling constant $C_\mu$. In fact, we may take $C = 5C_\mu^9$. We see that for any $x \in X$,
\[
\underset{a\in A}{\sum} \phi_a(x)=1.
 \]
 We define $f: X \rightarrow \mathbb{R}$ by :
\[
f(x): = \sum_{a \in A}\tilde{f}(\tilde{a})\phi_a(x).
\]
We consider the pointwise upper Lipschitz constant function on $X$ defined by
\[
\Lip f(x):= \limsup_{r\rightarrow 0} \sup_{y \in B_X(x,r)} \frac{|f(x) - f(y)|}{r}
\]

It can be shown that $\Lip f(x)$ is an upper gradient of $f$ provided that $f$ is locally Lipschitz (see Theorem 6.1 in \cite{C}). It clear from the construction that $f$ is locally Lipshitz. We will show that for all $a \in A$ such that $x \in B_{X}(a,\epsilon)$ ,
\[
\Lip f(x) \leq C  \sum_{\tilde{a} \sim \tilde{b}} \frac{|\tilde{f}(\tilde{a}) - \tilde{f}(\tilde{b})|}{\epsilon} = C \left|\grad \tilde{f}(\tilde{a})\right|,
\]
where $\frac{C}{\epsilon}$ is the Lipschitz constant of the $\phi_a$ functions. From this we glean a lower bound for the right half of (\ref{newdpi}).

By the maximality of $A$, for any $x \in X$ there is some $a_0 \in A$ such that $x \in B_{X}(a_0, \epsilon)$. Since this ball is open,
we assume that $r$ is small enough such that we may only consider points $y \in B_{X}(x, r)\subset B_{X}(a_0,\epsilon)$, i.e. $r< \frac{\epsilon - d_X(x, a_0)}{2}$. Let $D_x = \{a\in A: d_X(x,a) < 2\epsilon \}$ and $D_y = \{ a \in A :  d_X(y,a) < 2\epsilon \}$. Let $D = D_x \cup D_y$, which ultimately depends on $y$, and note that $D \subset \{a \in A :  d_X(a, a_0) < 3\epsilon \}$. We now show a useful pointwise bound for Lip $f(x)$. Observe that if $a \in A\backslash D_x$ then $\phi_a(x) = 0$ and if $a\in A\backslash D_y$ then $\phi_a(y) = 0$. Hence,
\begin{align*}
\frac{|f(x) - f(y)|}{r} & = \frac{1}{r} \left|\sum_{a\in D_{x}} \tilde{f}(\tilde{a})\phi_a(x) - \sum_{a\in D_{y}} \tilde{f}(\tilde{a})\phi_a(y) \right|\\
& = \frac{1}{r} \left| \sum_{a\in D} \tilde{f}(\tilde{a})\phi_a(x) - \sum_{a\in D}\tilde{f}(\tilde{a})\phi_a(y) - \sum_{a\in D} \tilde{f}(\tilde{a}_0)\phi_a(x) + \sum_{a \in D}\tilde{f}(\tilde{a}_0)\phi_a(y) \right| \\
\end{align*}
The equality in the second line is due to the fact that $\sum_{a\in D} \phi_a(x) = 1 = \sum_{a\in D}\phi_a(y)$.  After grouping like terms from the above, we continue:
\begin{align*}
\frac{|f(x) - f(y)|}{r} =  &\frac{1}{r} \left| \sum_{a \in D} \tilde{f}(\tilde{a})(\phi_a(x) - \phi_a(y)) - \sum_{a \in D} \tilde{f}(\tilde{a}_0)(\phi_a(x) - \phi_a(y)) \right| \\
 = &\frac{1}{r} \left| \sum_{a \in D} (\tilde{f}(\tilde{a}) - \tilde{f}(\tilde{a}_0))(\phi_a(x) - \phi_a(y)) \right| \\
 \leq &\frac{C}{r\epsilon} \sum_{a \in D} | \tilde{f}(\tilde{a}) - \tilde{f}(\tilde{a}_0) |  d_X(x,y) \\
 \leq &\frac{C}{\epsilon} \sum_{\tilde{a} \sim \tilde{a}_0} |\tilde{f}(\tilde{a})  - \tilde{f}(\tilde{a}_0)|.
\end{align*}

We may now conclude that if $x \in B(a_0, \epsilon)$ for some $a_0 \in A$, then there is a constant $C$ that depends only on $C_\mu$ such that
\begin{equation}
\label{pwgineq} \Lip f(x) \leq C \left|\grad \tilde{f}({\tilde{a}_0})\right|.
\end{equation}
We use this pointwise estimate to compare $L^p$ estimates of the gradients, in preparation for the Poincar\'e inequality.  Let $B_X(x,r)$ be a ball in $X$.  Using the results of Proposition \ref{blipcomp} and (\ref{ballLipeq}), we see that
\begin{equation}
\begin{aligned}
\int_{B_X(x,r)} (\Lip f)^p \,d\mu
& \leq \sum_{a \in A \cap B_X(x,r+\epsilon)} \left( \int_{B_X(a,\epsilon)} (\Lip f)^p \,d\mu \right)\\
& \leq  C^p \sum_{a \in A \cap B_X(x,r+\epsilon)} \left( \int_{B_X(a, \epsilon)}  \left|\grad \tilde{f}({\tilde{a}})\right|^p \, d\mu \right)\\
& =  C^p \sum_{a \in A \cap B_X(x,r+\epsilon)} \left|\grad \tilde{f}({\tilde{a}})\right|^p m(\tilde{a})\\
& \leq C^p \int_{B_V(\tilde{a}_0, (L+1)(r+2\epsilon))}\left|\grad \tilde{f}({\tilde{a}})\right|^p dm(\tilde{a}).
\end{aligned}
\label{LipGradInt}
\end{equation}

 With the above we now approach (\ref{newdpi}). Let $a_0 \in A$ be a nearest point to $x$ in $A$. Note that if $r<\epsilon$ the discrete Poincar\'e inequality (\ref{newdpi}) is trivially valid, so we can now say for all $r>0$ {by Lemma \ref{mcompmu}}, 
\begin{equation}
\label{ftildefPI}
\dashint_{B_{X}(x,r)} |f - f_{B_{{X}}(x,r)}|\, d\mu \leq Cr\left( {\avint_{B_V(\tilde{a}_0, {6\lambda}Lr)}\left|\grad \tilde{f}(\tilde{x})\right|^p dm(\tilde{x}) }\right)^\frac{1}{p} .
\end{equation}
Notice the radius for the average integral on the right hand side is $6\lambda Lr = (3\cdot 2L)\lambda r$. The $3$ appears from the assumption that $r \geq \epsilon$, and $2L$ appears from both the fact that $L \geq 1$ and from the constant in Lemma \ref{mcompmu}.

We now wish to verify the remaining part of the Poincar\'{e} inequality, i.e. replacing the left hand side of the above inequality with one related to the discrete function $\tilde f$. Instead of looking for the left hand side of (\ref{newdpi}) above, we search for {
\begin{equation} \label{AltPI}
\avint_{\tilde{z}\in B_V} \avint_{\tilde{w}\in B_V} |\tilde{f}(\tilde{z}) - \tilde{f}(\tilde{w})| dm(\tilde{w})dm(\tilde{z}).
\end{equation}}
This is a valid substitution since for any metric measure space $(X, d_X, \mu)$, and any measurable function $u$, the following two properties hold:
\begin{align*}
\dashint_B\dashint_B |u(x) - u(y)| d\mu(y)d\mu(x) &= \dashint_B\dashint_B |u(x) - u_B + u_B - u(y)| d\mu(y) d\mu(x) \\
& \leq 2\dashint_B |u(x) - u_B|d\mu(x), \\
\text{and} \ \ \ \dashint_B|u(x) - u_B| d\mu(x) & = \dashint_B\left|\dashint_B (u(x) - u(y))d\mu(y) \right| d\mu(x) \\
& \leq \dashint_B\dashint_B |u(x) - u(y)| d\mu(y)d\mu(x).
\end{align*}

Given $\tilde{a} \in V$ and $r> 0$, we look at the ball $B_V(\tilde{a}, r) \subset V.$ We fix two points $\tilde{z}, \tilde{w} \in B_V(\tilde{a}, r).$ We note that $B_X(z, \frac{\epsilon}{2}) \cap A = \{z\} \subset X$, and $B_X(w, \frac{\epsilon}{2})~\cap~A =\{w\}\subset X,$ by the { $\epsilon$-seperability} of $A$. Let $x,y$ be elements of $B_X(z, \frac{\epsilon}{2})$, and $B_X(w, \frac{\epsilon}{2})$ respectively. Recalling a useful fact about the $\phi$ functions from the partitions of unity, we may write
\begin{align*}
\tilde{f}(\tilde{z}) & = \sum_{b \in A} \tilde{f}(\tilde{z})\phi_b(x) \\
& = \sum_{b \in A}\tilde{f}(\tilde{z})\phi_b(x) + \left( f(x) - \sum_{b\in A} \tilde{f}(\tilde{b})\phi_b(x) \right)\\
& = \sum_{b \in A} (\tilde{f}(\tilde{z}) - \tilde{f}(\tilde{b}))\phi_b(x) + f(x).
\end{align*}
Similarly, we write
\[
\tilde{f}(\tilde{w}) = \sum_{b\in A}(\tilde{f}(\tilde{w}) - \tilde{f}(\tilde{b}))\phi_b(y) + f(y).
\]
Thus,
\begin{align*}
|\tilde{f}(\tilde{z}) - \tilde{f}(\tilde{w})| & = \left|\sum_{b\in A}(\tilde{f}(\tilde{z}) - \tilde{f}(\tilde{b})) \phi_b(x) + f(x) - \sum_{b\in A}(\tilde{f}(\tilde{w}) - \tilde{f}(\tilde{b}))\phi_b(y) - f(y)\right| \\
& \leq |f(x) - f(y)| + \sum_{b\in A} |\tilde{f}(\tilde{z}) - \tilde{f}(\tilde{b})|\phi_b(x) + \sum_{b \in A} |\tilde{f}(\tilde{w}) - \tilde{f}(\tilde{b})|\phi_b(y).
\end{align*}

Since $\phi_b(x) = 0$ whenever $x \notin B_X(b, 2\epsilon)$, and because $x~\in~B_X(z,\epsilon/2)$, then the sum from the second term can be taken over all $b\in A$ such that $d_X(b,z)~<~\frac{5\epsilon}{2}$, which means that we may instead just sum over neighbors:

\[ |\tilde{f}(\tilde{z})-\tilde{f}(\tilde{w})| \leq |f(x)-f(y)| + \sum_{\tilde{b} \sim \tilde{z}} |\tilde{f}(\tilde{z})-\tilde{f}(\tilde{b})| \phi_b(x) + \sum_{\tilde{b} \sim \tilde{w}} |\tilde{f}(\tilde{w})-\tilde{f}(\tilde{b})| \phi_b(y) \]

We now turn our sights back onto the double sum form of the left hand side of the Poincar\'{e} inequality. Using the above comparisons, and recalling that $x := x_{\tilde{z}}$ and $y:= y_{\tilde{w}}$ depend on $\tilde{z}$ and $\tilde{w}$, respectively,  we see that {
\begin{equation}
	\begin{aligned}
		 \int_{B_V} \int_{B_V} |\tilde{f}(\tilde{z}) - & \tilde{f}(\tilde{w})|  dm(\tilde{w})dm(\tilde{z}) \\
		& \leq \int_{B_V} \int_{B_V} |f(x_{\tilde{z}}) - f(y_{\tilde{w}})|dm(\tilde{w})dm(\tilde{z}) \\
		&\hspace{.5cm}+  \int_{B_V} \int_{B_V}\sum_{\tilde{z}\sim\tilde{b}} |\tilde{f}(\tilde{z}) - \tilde{f}(\tilde{b})|dm(\tilde{w})dm(\tilde{z})  \\
		&\hspace{.5cm}+  \int_{B_V} \int_{B_V}\sum_{\tilde{w}\sim\tilde{b}} |\tilde{f}(\tilde{w}) - \tilde{f}(\tilde{b})|dm(\tilde{w})dm(\tilde{z}).
	\end{aligned}
\label{3partsnoav}
\end{equation}}

We work with these terms on the right hand side separately, first with
\begin{equation}
\label{I} \int_{B_V} \int_{B_V} |f(x_{\tilde{z}}) - f(y_{\tilde{w}})|dm(\tilde{w})dm(\tilde{z}).
\end{equation}

By using the doubling property in the third line and the results of (\ref{blipcomp}) in the last line, we see that
\begin{align*}
 \int_{B_V} & \int_{B_V} |f(x_{\tilde{z}}) - f(y_{\tilde{w}})| dm(\tilde{w})dm(\tilde{z}) \\
 = &\sum_{\tilde{z} \in B_V} \sum_{\tilde{w} \in B_V} |f(x_{\tilde{z}}) - f(y_{\tilde{w}})|\mu(B(w, \epsilon))\mu(B(z,\epsilon)) \\
 \leq & C_\mu^2 \sum_{\tilde{z} \in B_V} \sum_{\tilde{w} \in B_V} |f(x_{\tilde{z}}) - f(y_{\tilde{w}})|\mu(B(w, \epsilon/2))\mu(B(z,\epsilon/2)) \\
 \leq &C_\mu^2 \sum_{\tilde{z}, \tilde{w} \in B_V} \int_{B_X(z, \frac{\epsilon}{2})} \int_{B_X(w,\frac{\epsilon}{2})} |f(x) - f(y)| \chi_{B(w,\frac{\epsilon}{2})}(y) \chi_{B(z,\frac{\epsilon}{2})}(x) d\mu(y) d\mu(x) \\
 \leq& C_\mu^2 \int_{B_X(a, (L+2)r)} \int_{B_X(a, (L+2)r)} |f(x) - f(y)| d\mu(y) d\mu(x),
\end{align*}
{where $\chi_A$ as usual stands for the characteristic function of $A \subset X$.}
Now for  the second term of our inequality from (\ref{3partsnoav}):
\begin{align*}
 \int_{B_V} \int_{B_V} \sum_{\tilde{b} \sim \tilde{z}} | & \tf(\tilde{z}) - \tf(\tilde{b})| dm(\tilde{w})dm(\tilde{z}) \\
= &\int_{B_V} \left[ \int_{B_V} \sum_{\tilde{z} \sim \tilde{b}} |\tf(\tilde{z}) - \tf(\tilde{b})| dm(\tilde{z}) \right] dm(\tilde{w}) \\
= &\epsilon \cdot \int_{B_V} \left[ \int_{B_V} \left|\grad \tf(\tilde{z})\right| dm(\tilde{z}) \right] dm(\tilde{w}) \\
= &\epsilon \cdot m(B_V(\tilde{a},r)) \int_{B_V} \left|\grad \tf(\tilde{z})\right| dm(\tilde{z})
\end{align*}

Clearly the same {quantity} can be used to bound the third term of the summation by transposing $\tilde{z}$ with $\tilde{w}$. Summarizing, from (\ref{3partsnoav}) we achieve
\begin{equation*}
	\begin{aligned}
\int_{B_V} \int_{B_V} |\tilde{f}(\tilde{z}) &- \tilde{f}(\tilde{w})|  dm(\tilde{w})dm(\tilde{z}) \leq\\
& C_\mu^2 \int_{B_X(a, (L+2)r)} \int_{B_X(a, (L+2)r)} |f(x) - f(y)| d\mu(y) d\mu(x)\\
& + 2 \epsilon \cdot m(B_V(\tilde{a},r)) \int_{B_V} \left|\grad \tf(\tilde{z})\right| dm(\tilde{z})
	\end{aligned}
\end{equation*}
By Lemma \ref{mcompmu} and {Remark \ref{drmk}} we are free to average all these integrals to obtain
\begin{equation}
	\begin{aligned}
\avint_{B_V} \avint_{B_V} |\tilde{f}(\tilde{z}) &- \tilde{f}(\tilde{w})|  dm(\tilde{w})dm(\tilde{z}) \leq\\
& C \avint_{B_X(a, (L+2)r)} \avint_{B_X(a, (L+2)r)} |f(x) - f(y)| d\mu(y) d\mu(x)\\
& + 2 \epsilon \avint_{B_V} \left|\grad \tf(\tilde{z})\right| dm(\tilde{z}),
	\end{aligned}
	\label{2partsav}
\end{equation}
where $C$ is a constant that depends only on $C_\mu$.
We now apply the Poincar\'e inequality version (\ref{ftildefPI}) on the first term on the right-hand side of inequality (\ref{2partsav}). Recalling the discussion after (\ref{AltPI}), we achieve
\begin{equation*}
	\begin{aligned}
\avint_{B_V} \avint_{B_V} |\tilde{f}(\tilde{z}) &- \tilde{f}(\tilde{w})|  dm(\tilde{w})dm(\tilde{z}) \leq\\
& C_1 (L+2)r \left(\avint_{B_V(a, 6\lambda L(L+2)r)} \left|\grad \tilde{f}(\tilde{x})\right|^p dm(\tilde{x}) \right)^{1/p} \\
& + 2 \epsilon \avint_{B_V} \left|\grad \tf(\tilde{z})\right| dm(\tilde{z}),
	\end{aligned}
\end{equation*}
for some constant $C_1$ depending on the data of the Poincar\'{e} inequality and $C_\mu$. Now, by employing H\"older's inequality and the assumption that $\epsilon < r$ on the second term on the right hand side we finally conclude:
\begin{equation*}
\avint_{B_V} \avint_{B_V} |\tilde{f}(\tilde{z}) - \tilde{f}( \tilde{w})|  dm(\tilde{w})dm(\tilde{z}) \leq C_2 r \left( \avint_{\lambda_1 B_V} \left|\grad \tilde{f}\right|^p dm \right)^{1/p}.
\end{equation*}
This is the desired Poincar\'e inequality.  The constants $C_2$ and $\lambda_1$ ultimately depend only on the data of the Poincar\'e inequality and the doubling constant of $X$.
Along with Section 4, the above shows that if a metric measure space, $(X, d_X, \mu)$, supports a $(1,p)$-Poincar\'{e} inequality and is doubling, then the discretization, $(V, d_V, m)$, supports a discrete $(1,p)$-Poincar\'{e} inequality and is doubling and Theorem \ref{XtoV} is proved.
\end{proof}

\section{Pointed Measured Gromov-Hausdorff convergence}

We would like to answer the converse question to what we have shown: If we have a sequence of graphs that support a Poincar\'{e} inequality, have a doubling measure, and ``converge'' to some base space $X$, does $X$ also support a Poincar\'{e} inequality and have a doubling measure $\mu$? To answer this question, we must of course discuss the sense of convergence. In this section we introduce definitions and results about Gromov-Hausdorff convergence that will aid in the answer to the above question. A reader well-versed in the Gromov-Hausdorff topology may safely skip the discussion below and pick up again at Theorems \ref{Cheeg1} and \ref{GHconv}.  For in-depth discussions of Gromov-Hausdorff convergence from a geometric point of view see the book \cite{BB} by D. Burago, Y. Burago, and S. Ivanov.  The upcoming book \cite{HKST} by J. Heinonen, P. Koskela, N. Shanmugalingam, and J. Tyson also discusses the Gromov-Hausdorff topology but from an analytic point of view.

We say that $(X,d,q)$ is a \textit{pointed metric space} if $(X,d)$ is a metric space, and $q\in X$. Let $A$ be a subset of $X$ and fix $\epsilon > 0$. We define the $\epsilon$-\textit{neighborhood} of A as the set
\[
N_\epsilon(A) := \{x\in X | d(x, A) < \epsilon\} = \bigcup_{a\in A} B(a, \epsilon).
\]\label{epsnbhd}

\begin{deff} \label{Hdist} Given a metric space $(Z,d)$ and $\epsilon > 0$, we define the {\textit{Hausdorff distance}} in $Z$ between non-empty sets $A,B \subset Z$  by
\[
d_H^Z(A,B) := \inf\{\epsilon > 0 : A\subset N_\epsilon(B) \text{ and } B \subset N_\epsilon(A)\}.
\]\end{deff}
Notice that if $A$ is a maximal $\epsilon$-net of $Z$, then $d_H^Z(A,Z) \leq \epsilon$.
\begin{deff}
A sequence of pointed separable metric spaces
\[
(X_1, d_1, q_1), (X_2, d_2, q_2)...
\]
is said to \textit{pointed Gromov-Hausdorff} converge to a pointed separable metric space $(X,d,q)$ if for each $r, \eta$ such that $0< \eta < r$, there exists a positive integer $i_0$ such that for each $i\geq i_0$ there exists a map $f_i: B(q_i, r) \rightarrow X$ that satisfies the following:
\begin{enumerate}
\item $f_i(q_i) = q$;
\item $|d(f_i(x), f_i(y)) - d_i(x,y)| < \eta$ for all $x,y \in B(q_i, r)$;
\item $B(q,r-\eta) \subset N_\eta(f_i(B(q_i,r)))$.
\end{enumerate}
\end{deff}
In this definition, the maps $f_i$ are dependent on $\eta$ but not required to be even continuous. We denote this convergence by $(X_i, d_i, q_i) \overset{GH}{\rightarrow} (X, d, q)$.

Let $(\mu_i)_{i=1}^\infty$ be a sequence of locally finite Borel measures on a metric space $X$. If for every boundedly supported continuous function $\phi: X \rightarrow \mathbb{R}$
\begin{align*}
\int_X \phi d\mu_i \rightarrow \int_X \phi d\mu
\end{align*}
as  $i \rightarrow \infty$, then the sequence $(\mu_i)$ is said to \textit{weak*,} or \textit{weak-star, converge} to $\mu$. We denote this convergence by $\mu_i \overset{*}{\rightharpoonup} \mu$. Suppose that $(X, \mu)$ is a metric measure space, and $Y$ is a metric space. Given a function $f: X \rightarrow Y$ we may define the \textit{push forward measure}, $f_{\#}\mu$, as follows. For $A \subset Y$ we let $f_{\#}\mu(A) := \mu(f^{-1}(A))$.

Let $(X_i, d_i, q_i, \mu_i)_{i=1}^{\infty}$ be a sequence of pointed metric measure spaces. If $(X_i, d_i, q_i)$ pointed Gromov Hausdorff converges to $(X,d, q)$ and if for every $r>0$ there exist isometric embeddings $\iota_i: \overline{B}(a_i,r) \rightarrow \ell^\infty$ such that $d_H^{\ell^\infty}(\iota_i(\overline{B}(q_i,r)), \iota(\overline{B}(q,r))) \rightarrow 0$ and $(\iota_i)_\#\mu_i\lfloor\overline{B}(q_i,r) \overset{*}{\rightharpoonup} \iota_\#\mu\lfloor\overline{B}(q,r)$ as measures on $\ell^\infty$, then we say that $(X, d, \mu)$ is a  \textit{measured Gromov-Hausdorff limit} of the sequence $(X_i, d_i, \mu_i)_{i=1}^{\infty}$. We denote this convergence by
\[
(X_i, d_i, q_i, \mu_i) \overset{GH}{\rightarrow} (X,d, q, \mu).
\]
 We note that in the context of this paper, each $X_i$ will be separable, and so by the work of M. Fr\'{e}chet (see \cite{F}), there will always exist some isometric embedding from $X_i$ to $\ell^\infty.$ 

With these definitions in place, we now present two theorems that will help us answer {the} guiding questions presented at the beginning of this section.
\begin{thm}[\cite{C},Theorem 9.1] \label{Cheeg1}
Let $(X_i, d_i, q_i, \mu_i)$ be a sequence of complete spaces which pointed measured Gromov-Hausdorff converge to a complete space $(X,d,q,\mu)$. If each of the measures $\mu_i$ is doubling with constant $C_D$, then $\mu$ is also doubling with constant $C_D$.
\end{thm}
The second theorem that will be of great importance was proved independently by J. Cheeger \cite{C} and S. Keith \cite{Kt}.
\begin{thm}[\cite{C},Theorem 9.6]
\label{GHconv}Let $(X_i, d_i, q_i, \mu_i)$ be a sequence of complete spaces that pointed measured Gromov-Hausdorff converge to a complete space $(X,d,q,\mu)$. Let $1\leq p<\infty, C_D, C_p < \infty$ and $\lambda \geq 1$ be fixed. If each of the measures $\mu_i$ is doubling with constant $C_D$, and each space $(X_i, d_i, \mu_i)$ satisfies the $(1,p)$-Poincar\'{e} inequality with constants $C_p$ and $\lambda$, then $(X, d, \mu)$ also satisfies the $(1,p)$-Poincar\'{e} inequality with constants $C'_p$ and $\lambda'$ depending only on $p, C_p$ and $C_D$.
\end{thm}
The above two theorems are usually presented with the requirement that $(X_i, d_i)$ be length spaces. However, this requirement is not necessary for the desired results. To see a discussion about the lack of length spaces we have presented, refer to \cite{HKST} (Chapter 11).
We now have the resources necessary to explicitly achieve Theorem~\ref{VtoX}.  We first discuss some examples to put the formulation of Theorem \ref{VtoX} in context.

\section{Examples from $\mathbb{R}^2$}
\label{sec:examples}

Here we wish to let $X= \mathbb{R}^2$, $d$ be the Euclidean metric, and $\mu$ be Lebesgue measure.  In the next section below we will consider sequences of connected spaces, derived from a sequence of discrete spaces, which converge under the pointed measured Gromov-Hausdorff  topology to a space {$(\tilde X,\tilde d)$}.  Here we wish to show that even when discretized versions of $\mathbb{R}^2$ are considered, the limit space will not be $\mathbb{R}^2$ with the Euclidean metric and Lebesgue measure.  This {is the reason for the flexibility of the conditions in Theorem \ref{VtoX}}.

Consider the discretization sequence generated by the integer grid $\mathbb{Z} \times \mathbb{Z}$ and dyadic scaling, i.e. $V_1 = \mathbb{Z} \times \mathbb{Z}$ and $V_i = {\epsilon_i} (\mathbb{Z} \times \mathbb{Z})$ where $\epsilon_i = \frac{1}{2^{i-1}}$.  Under the scheme for $(V_i, d_i, m_i)$ introduced in Section 3,  for large $n$
\[ d_n ((0,0), (0, 1)) = 1/2 \hspace{1cm}  d_n((0,0), (1,1)) = 1/2. \]
In fact, one can see that $d_n$ shrinks distances of neighbors by a factor of $2^n$ for horizontal and vertical neighbors, leaving the distance between two points along a vertical or horizontal line fixed as $n \rightarrow \infty$. However, the distance between neighbors along a diagonal also shrinks by $2^n$ under $d_n$, but the Lebesgue distance between two points shrinks by a rate of $(\sqrt{2}/2^n)$, leaving a distortion of $\sqrt{2}$. However,  verifying the discrete Poincar\'e inequality on each of the $V_i$ is a simple task as shown below, and the data is independent of $i$.

As far as the measures $m_i$ for this example are concerned, a similar problem occurs.  If one pushes these measures forward onto $\mathbb{R}^2$ we get a sequence of measures, $\mu_i$, which does not converge to Lebesgue measure.  For example
\[ \mu_3(B_{V_3}((0,0),1)) = \frac{43}{16} \pi, \mu_4(B_{V_4}((0,0), 1)) = \frac{193}{64} \pi, \mbox{ and } \mu_5(B_{V_5}((0,0),1)) = \frac{793}{256} \pi \]
and $\lim_{n \to \infty} \mu_n (B_{V_n}((0,0),1)) = \pi^2$.  This convergence can be seen geometrically as the vertices in the $n$-th discretization cover the unit ball with ``squares'' which have area $\pi$ due to the definition of $m_i$. In fact, $\mu_i$ will weak-star converge to the Lebesgue measure multiplied by $\pi$.  It should be clear at this point that a less symmetric discretization of $\mathbb{R}^2$ can lead to measures which are not just multiples of Lebesgue measure.  One of the novelties of Theorem \ref{VtoX} is that one need not calculate the weak$^*$ limit of $\mu_n$. Instead, it is enough to verify that the collection {of discretezations have the uniform properties (2) through (5)}.  In this case, the doubling constant of $\mu$ is $4$, and for all $i\in \N$ the maximum degree of any vertex in $V_i$ is $28$. We find that $7128= 28\cdot4^4$ suffices as the doubling constant for all $m_i$. Thus, {after verifying a discrete Poincar\'e inequality}, all the conditions of Theorem \ref{VtoX} are satisfied, and one can conclude that $(\R^2, d, \mu)$ must also support a Poincar\'e inequality with the similar data.


Now to verify this discerete Poincar\'e inequality for our discretization of $\mathbb{R}^2$. For each $i \in \N$ we set $V_i: = \frac1{2^i}\left(\Z \times \Z\right)$ where the step size between neighbors is $\frac1{2^i}$. That is, we are taking $\epsilon_i = \frac1{2^i}$.  We will consider a ball, $B\subset \R^2$, of radius $n \in N$ centered at the point $(0,0)$. Let $x$ and $y$ be two points in $B$. By the construction of $\Z \times \Z$, there is a path of points $p_0, p_1, \dots, p_k \in B$ such that $x = p_0 \sim p_1 \sim \dots \sim p_k = y$. To assign these points, let $\gamma$ be the straight line path in $\R^2$ from $x$ to $y$. We have that $\gamma$ is a rectifiable curve and assume that length$(\gamma) = |x-y|$. We set $t_i$ to be the point on $\gamma$ such that $x = t_0$, $t_k = y$, and for each $i$ we have that $|t_i - t_{i-1}| = \epsilon_i$, with the exception of $|t_k - t_{k-1}|$, which may be less than or equal to $\epsilon_i$. For each $t_i$ along $\gamma$, there is a point in $V_i$ within $\epsilon_i$ distance from $t_i$ on $\gamma$. We let $p_i$ be these points, and see that $|p_i - p_{i-1}| \leq 3\epsilon_i$. Thus, for each $i$ we have that $p_i \sim p_{i-1}$. By the triangle inequality, we see that
\[
|f(x) - f(y)| \leq \sum_{i =1}^{k} |f(p_i) - f(p_{i-1})|.
\]

We will integrate both sides of the above inequality. Such integration, along with the observation above yields
\[
\sum_{x \in B} |f(x) - f(y)| m_i(x) \leq Cn \sum_{x \in B} \sum_{z \sim x} |f(z) - f(x)| m_i(x),
\]
where $C$ is a constant depending on the doubling constant of $m_i$. Since each $m_i$ has a doubling constant that is uniform across all $i\in \N$, then this $C$ is uniform among all $i$ as well. In fact, we may take $C$ to be $256$, which is the doubling constant of $\mu$ to the fourth power. Integrating again on both sides we see that
\[
\sum_{y \in B}\sum_{x \in B} |f(x) - f(y)| m_i(x)m_i(y) \leq Cn\,m_i(B) \sum_{x \in B} \sum_{z \sim x} |f(z) - f(x)| m_i(x).
\]
By averaging both the two summations on the left hand side of the inequality, we arrive at a $(1,1)$-Poincar\'e inequality:
\[
\avint_B \avint_B |f(x) - f(y)| dm_i(x) dm_i(y) \leq C n \avint_B |\grad f(x)| dm_i(x).
\]
Through use of H\"older's inequality we arrive at the desired $(1,p)$-Poincar\'e inequality:
\[
\avint_B \avint_B |f(x) - f(y)| dm_i(x) dm_i(y) \leq C n \left(\avint_B |\grad f(x)|^p dm_i(x)\right)^{\frac1p}.
\]

Since we have that $C$ is uniform constant independent of $\epsilon_i$, then by Theorem \ref{VtoX}, $(\R^2, d, \mu)$ will support a $(1,p)$-Poincar\'e inequality. The ease of the above argument displays the usefulness of Theorem \ref{VtoX}.

\section{Proof of Theorem \ref{VtoX}}

Recall that we begin with a doubling complete metric measure space $(X, d_X, \mu)$.
We begin by supposing that $(X,d_X, \mu)$ supports a $(1,p)$-Poincar\'e inequality. By Theorem (\ref{XtoV}), we can make a nested embedded sequence of graphs $(V_i, d_{V_i}, m_i)$ into $X$ {with $\epsilon_i = 1/2^{i-1}$}. By construction, we see that
\[
d_H(n_i(V_i), X) \leq \epsilon_{V_i} \to 0
\]
as $i \to \infty$. Lemma \ref{blipcomp} guarantees that $d_{V_i}$ is bi-Lipschitz equivalent to $d_X$ for all $x, y \in V_i$. Lemma \ref{mcompmu} showed that $\mu$ is comparable to $m_i$ on balls with radius greater than $\epsilon_{V_i}$. Section 4 showed that $m_i$ was doubling, with doubling constant independent of $\epsilon_{V_i}$, ensuring that the entire family $(V_i, d_{V_i}, m_i)$ has a uniform doubling constant. Section 5 showed that all $(V_i, d_{V_i}, m_i)$ support a $(1,p)$-Poincar\'e inequality with uniform constants. Thus, one direction of the theorem is proved.

Conversely, assume that there exists a nested embedded sequence of graphs $(V_i, d_{V_i}, m_i)$ into $X$ such that
\begin{enumerate}
\item[(1)] The Hausdorff distance, $d_H (n_i(V_i), X) = H_i$, is finite for all $i \geq 0$ and $H_i \to 0$ as $i \to \infty$.
\item[(2)] There is a uniform $L > 1$ such that for all $i \geq 0$ and all $x,y \in V_i$,
\[ \frac{1}{L} d (n_i(x), n_i(y)) \leq d_{V_i} (x,y) \leq L \, d (n_i(x), n_i(y)) \]
\item[(3)] There is a uniform $K > 1$ such that for all $i \geq 0$ all $r > H_i$ and $x \in V_i$
\[ \frac1K \leq \frac{m_i (B_{V_i}(x,r))}{\mu (B_X (x,r))} \leq K\]
\item[(4)] $(V_i, d_{V_i}, m_i)$ are all doubling with uniform doubling constant.
\item[(5)] $(V_i, d_{V_i}, m_i)$ all support a $(1,p)$-Poicar\'e inequality with uniform data.
\end{enumerate}
We must now show that $(X,d_X, \mu)$ supports a $(1,p)$-Poincar\'e inequality. Our method of verifying this result requires the transformation of the sequence of graphs, $(V_i)$ into a new sequence of connected topological spaces. This is necessary to use Theorems \ref{Cheeg1} and \ref{GHconv} to verify the $(1,p)$-Poincar\'e inequality on $X$.

\subsection{Extension of $V_i$} We may extend each $V_i$ into a path connected space, $G_i$, in a similar manner as was done by N. Shanmugalingam in \cite{N} (Section 3), by placing an isomorphic copy of the interval $[0, \epsilon_{V_i}]$ between each set of neighbors. These extensions are not graphs in the traditional sense, but rather, they are a connected space that may be thought of as 1-simplexes. They are connected in the topological sense, and we will often speak of the vertices in $G_i$, and points on its {\em{edges}}, which are points in an interval $[0, \epsilon_{V_i}]$ with specified associated vertices. We also adjust the metric and measure for $G_i$. For the metric, we say that two vertices $\tilde{x}, \tilde{y}$ in $V_i \subset G_i$ have a distance $\tilde{d}_{G_i}(\tilde{x},\tilde{y}) := d_i(\tilde{x},\tilde{y}) $. If $\tilde{x}$ is a vertex and $\tilde{y}$ is on an edge which is connected to $\tilde{x}$, we use the obvious distance inherited from the isomorphic copy of $[0, \epsilon_{V_i}]$.  If $\tilde{x}$ is a point on the edge and $\tilde{y}$ is a vertex which is not an endpoint for the edge that $\tilde{x}$ is on, we select the vertex $\tilde{v}_{\tilde{x}}$ on the edge that $\tilde{x}$ is on which minimizes the following expression: $\tilde{d}_{G_i} (\tilde{x}, \tilde{y}) := d_{V_i} (\tilde{y}, \tilde{v}_{\tilde{x}}) + |\tilde{x} - \tilde{v}_{\tilde{x}}|$.  If both $\tilde{x}$ and $\tilde{y}$ are on edges, we just extend as in the previous case in the obvious way. We note that when finding the distance between two points on an edge, it may not be the case that distance is found by selecting the closest vertices to the points, but becomes the minimum distance of 4 different paths that involve the associated vertices. This defines a new metric, $\tilde d_{G_i}$, on our space $G_i$.

We also build a new measure $\overline{m}_i$ in terms of $m_i$ such that $\overline{m}_i$ is comparable to $m_i$ on balls with radius greater than $\epsilon_{V_i}$. This implies that $\overline{m}_i$ is comparable to $\mu$ on balls greater than $\epsilon_{V_i}$ just as $m_i$ is by assumption (3). If $U$ is a subset of $G_i$, then we define
\[
\overline{m}_i(U) = \sum_I \frac{\text{length}(I \cap U)}{\epsilon_{V_i}}\left[m_i(\tilde{x}_I) + m_i(\tilde{y}_I)\right],
\]
where the sum is over the intervals $I=[0, \epsilon_{V_i}]$ such that $I\cap U \neq \varnothing$, and each $I$ has associated endpoints $\tilde{x}_I$ and $\tilde{y}_I$. Length is, of course, understood as Lebesgue measure. {Because of the uniform bound on degree in the graph $V_i$ necessitated by the doubling condition property also (3) holds for $\overline{m}_i$.}  We now want to verify that this extended space $(G_i, \tilde d_{G_i}, \overline {m}_i)$ has a doubling measure, and supports a $(1,p)$-Poincar\'e inequality in the traditional sense of (\ref{pi}). Similar proofs to the next two lemmas can be found in \cite{N} for the case that $\epsilon_{V_i} = 1$ and $m_i(\tilde x) = 1$ for all $\tilde x \in V_i \subset G_i$. In this paper, we present the more general case that includes more general measures and metrics.

\begin{lem} \label{mgdoub}
The measure $\overline{m}_i$ is doubling.
\end{lem}

\begin{proof}
We first consider that $\tilde{x} \in V_i$, and will discuss $\tilde{x}$ being a point on an edge shortly.

\textit{Case 1:} If $r \leq \epsilon_{V_i}$, then
\[
\frac{r}{\epsilon_{V_i}}m_i(B(\tilde{x},r)) = \frac{r}{\epsilon_{V_i}}m_i(\tilde{x})\leq  \frac{r}{\epsilon_{V_i}}\sum_I [m_i(\tilde{x}) + m_i(\tilde{y}_I)] = \overline{m}_i(B(\tilde{x}, r)).
\]
Remembering that $\tilde x$ has a bounded degree that only depends on the doubling constant of $m_i$, we see that
\[
\overline{m}_i(B(\tilde{x}, 2r)) \leq \sum_I \frac{r}{\epsilon_{V_i}}m_i(B(\tilde{x}, 2\epsilon_{V_i})) \leq \frac{C^2 r}{\epsilon_{V_i}}m_i(B(\tilde{x}, \epsilon_{V_i})) = \frac{C^2 r}{\epsilon_{V_i}}m_i(B(\tilde{x}, r)),
\]
where $C$ is a constant that only depends on the doubling constant of $m_i$. Combining these two facts, we see that
\[
\overline{m}_i(B(\tilde{x},2r)) \leq \frac{2C^2r}{\epsilon_{V_i}}m_i(B(\tilde x,2r)) \leq 2C^3\overline m_i(B(\tilde{x},r)).
\]

\textit{Case 2:} If $r \geq \epsilon_{V_i}$, then we notice that the largest length of any edge in $B(\tilde x, r)$ is $\epsilon_{V_i}$.
We see that
\begin{align*}
\overline{m}_i(B(\tilde{x}, r)) &= \sum_I \frac{\text{length}(I \cap U)}{\epsilon_{V_i}}\left[m_i(\tilde{z}_I) + m_i(\tilde{y}_I)\right] \\
& \leq \sum_I \left[m_i(\tilde{z}_I) + m_i(\tilde{y}_I)\right] \\
& \leq 2C^4\ m_i(B(\tilde{x}, r + \epsilon_{V_i})) \\
& \leq 2C^5\ m_i(B(\tilde{x}, r)).
\end{align*}
Recall that the max degree of any vertex in $V_i$ is less than or equal to $C^4$.   { It is trivial to see that $m_i(B(\tilde{x},r)) \leq \overline{m}_i (B(\tilde{x},r))$ in this case, and so $m_i$ and $\overline{m}_i$ are comparble, and hence $\overline{m}_i$ is doubling.

Now, we consider the case that $\tilde{x}$ is not in $V_i$, this is a little less clear.  Let $\tilde v$ be the nearest vertex to $\tilde x$.  We have a few cases:

\textit{Case 1:} If $r/2 \geq \epsilon_{V_i}$, we note that
\[ B(\tilde v, r- \epsilon_{V_i}) \subset B(\tilde x, r) \subset B(\tilde v, r + \epsilon_{V_i}) \]
and so
\[ \overline{m}_i (B (\tilde x, 2r)) \leq \overline{m}_i (B(\tilde v, 2r + \epsilon_{V_i})) \leq C \overline{m}_i (B(\tilde v, \frac{r}{2})) \leq C \overline{m}_i (B(\tilde v, r - \epsilon_{V_i})) \leq C \overline{m}_i (B(\tilde x, r)) \]
where $C$ depends on the doubling constant of $m_i$.

\textit{Case 2:} If $\epsilon_{V_i} < r < 2 \epsilon_{V_i}$, recall $\tilde d_{G_i}(\tilde x, \tilde v) \leq \epsilon_{V_i}/2$, and that
\[
B(\tilde x, 2r) \subset B(\tilde v, 2r + \epsilon_{V_i}) \subset B(\tilde v, 3r).
\]
 Hence,
\[
 \overline{m}_i (B(\tilde x, 2r)) \leq \overline{m}_i (B(\tilde v, 3r)) \leq C \overline{m}_i (B(\tilde v, r/4)),
 \]
where again $C$ depends only on the doubling constant of $m_i$.  By the assumption in this case, $B(\tilde v, r/4) \subset B(\tilde x, r)$ and so $\overline{m}_i(B(\tilde x, 2r) \leq C \overline{m}_i (B(\tilde x, r))$.

\textit{Case 3:} If $\epsilon_{V_i}/4 < r < \epsilon_{V_i}$, we note that $m_i(\tilde v)$ is comparable to $m_i(\tilde w)$ for any $\tilde w$ at distance $2 \epsilon_{V_i}$ or less away from $\tilde v$.  This is because of the doubling of $m_i$.  Then $\overline{m}_i (B(\tilde x, 2r))$ is bounded above a constant times the sum of all the $m_i$ measures of vertices $\tilde w$ at distance $2 \epsilon_{V_i}$ or less away from $\tilde v$, which is comparable to $m_i(\tilde v)$.  But $\overline{m}_i (B(\tilde x, r)) \geq m_i (\tilde v) /4$ because $r > \epsilon_{V_i}/4$ and the definition of $\overline{m}_i$.  So doubling follows.

\textit{Case 4:} If $r < \epsilon_{V_i}/4$, this case is trivial because $B(\tilde x, 2r)$ can only contain edges connected to $\tilde v$.
}
\end{proof}

\begin{lem}
If $(V_i,d_i,m_i)$ supports a discretized $(1,p)$-Poincar\'{e} inequality in the sense of $(\ref{dpi})$, then $(G_i, \tilde d_{G_i}, \overline{m}_i)$ supports a $(1,p)$-Poincar\'{e} inequality in the sense of $(\ref{pi})$.
\end{lem}

\begin{proof}
{Since $G_i$ is a complete space, and by Lemma \ref{mgdoub} we have that $\overline m_i$ is doubling, then it suffices to verify this lemma with Lipschitz functions (see Theorem 8.4.2 in~\cite{HKST}). Let $u: G_i \to \R$ be a Lipschitz function, and recall that $\Lip u$ is an upper gradient of $u$.} We will assume that our ball, $B$, is centered at a vertex, $\tilde{x}$. We may do this due to the fact that we may increase the $\lambda$ value in the data of a Poincar\'e inequality by a constant that does not depend on $r$. 

\textit{Case 1:} The radius $r \leq \epsilon_{V_i}$.  For $\tilde{y}$ a neighbor of $\tilde{x}$, let $r\tilde{y}$ represent a point on the interval connecting $\tilde{x}$ to $\tilde{y}$ with a distance of $r$ from $\tilde{x}$. {We see that
\begin{equation}
|u(s\tilde{y}) - u(\tilde{x})| \leq \int_{\tilde x}^{s\tilde y} |\Lip u(\tau\tilde y)| d\tau,
\label{ug}
\end{equation}
{where $d\tau$ is Lebesgue measure on $[0, \epsilon_{V_i}]$.}  We notice the following bound, which proves useful in later calculations. Let $c \in \R$ and suppose that $c \leq u_B$. Then we see that
\begin{align*}
\avint_B |u_B - c|d\overline m_i = u_B -c = \avint_B u\,d\overline m_i - \avint_B c\,d\overline m_i = \avint_B (u - c)d\overline m_i \leq \avint_B |c -u|d\overline m_i.
\end{align*}
If $c > u_B$, then we have the similar result that
\[
\avint_B |u_B - c|d\overline m_i = c - u_B  \leq \avint_B |c -u|d\overline m_i.
\]
Notice that $|u(\tilde y) - u_B| \leq |u(\tilde y) - c| + |c - u_B|$ for each $\tilde y \in G_i$. Hence,
\[
\avint_B |u- u_B|\, d\overline m_i \leq \avint_B |u-c| d\overline m_i + \avint_B|c-u_B|d\overline m_i \leq 2\avint_B|c-u|d\overline m_i,
\]
for any $c \in \R$. In particular, we see that
\begin{equation}
\avint_B |u- u_B|\, d\overline m_i \approx \inf_{c \in \R} \avint_B|c-u|d\overline m_i.
\label{BltPI}
\end{equation}

Notice that $u(\tilde x)$ is a constant value, since $\tilde x$ is a fixed point in $B \subset G_i$. Then,
\[
\avint_B |u -u_B|\, d\overline m_i \leq 2\avint_B|u(\tilde x) - u|d\overline m_i.
\]
This right hand value is what we will use to show our Poincar\'{e} inequality. Recalling (\ref{ug}), we see that

\begin{align*}
\avint_B|u(\tilde x) - u|\, d\overline m_i& = \frac 1{\overline m_i(B)} \int_B |u(\tilde x) - u|\, d\overline m_i \\
& = \frac1 {\overline m_i(B)} \sum_{\tilde y \sim \tilde x} \int_{\tilde x}^{r\tilde y}|u(\tilde x) - u(s\tilde y)|\, d\overline m_i(s\tilde y) \\
& \leq \, \frac1 {\overline m_i(B)} \sum_{\tilde y \sim \tilde x} \int_{\tilde x}^{r \tilde y}\left( \int_{\tilde x}^{s \tilde y}|\Lip u(\tau \tilde y)|d\tau \right) d\overline m_i(s\tilde y) \\
& \leq \frac1 {\overline m_i(B)} \sum_{\tilde y \sim \tilde x} \int_{\tilde x}^{r \tilde y} \int_{\tilde x}^{r \tilde y}|\Lip u(\tau \tilde y)|d\tau d\overline m_i(s\tilde y).
\end{align*}
{We notice that $d \overline m_i (\tau y) = \frac{m_i(\tilde x) + m_i(\tilde y)}{\epsilon_{V_i}} d \tau$, and so continuing}

\begin{align*}
\avint_B|u(\tilde x) - u|d\overline m_i &  = \frac {r}{\overline m_i(B)} \sum_{\tilde y \sim \tilde x} \int_{\tilde x}^{r\tilde y} |\Lip u(\tau\tilde y)|d\overline m_i(\tau \tilde y) \\
& = r\avint_B |\Lip u| d\overline m_i.
\end{align*}
Now, by using H\"older's inequality, we have a $(1,p)$-Poincar\'e inequality.

\textit{Case 2:} The radius $r > \epsilon_{V_i}$.
We define a new function $\hat u$ as the restriction of $u$ to the vertex set $V_i$. Now, consider the piecewise linear extension of $\hat u$ to $G_i$, which we will denote $\tilde u$. Note that $\tilde u$ and $u$ agree on $V_i \subset G_i$ and differ on the edge set of $G_i$. Let $f = \tilde u - u$, which vanishes on $V_i$ {and note that $f_B = \tilde u_B - u_B$}. Now we consider our Poincar\'e inequality using the alternate form from (\ref{BltPI}) for the right hand side. Clearly we see that
\begin{equation}
\begin{aligned}
\avint_B |u - u_B|d\overline m_i & \leq \avint_B|\tilde u - \tilde u_B|d\overline m_i + \avint_B|f - f_B|d\overline m_i.
\label{split}
\end{aligned}
\end{equation}
We will investigate the two terms on the right hand separately. First, we consider
\[
\avint_B|\tilde u - \tilde u_B|d\overline m_i.
\]
By the proof of Lemma $\ref{mgdoub}$, $\overline{m}_i (B(\tilde y, \epsilon_{V_i}))$ is comparable to $m_i(\tilde y)$ for all $\tilde y \in V_i$. {First, we note that}
\[
\Lip \tilde u(\tilde z) = \frac{|\tilde u(\tilde x) - \tilde u(\tilde y)|}{\epsilon_{V_i}}
\]
whenever $\tilde z$ is in the edge connecting vertices $\tilde x$ and $\tilde y$,  {and by inequality (\ref{dpi})}

\begin{align*}
\frac{1}{m_i(B)}\sum_{\tilde{z} \in B} |\tilde u(\tilde{z}) - \tilde u_B|m_i(\tilde z) &\leq Cr\left(\frac{1}{m_i(\lambda B)}\sum_{z\in \lambda B}|\grad  \tilde u(\tilde z)|^p m_i(\tilde{z}) \right)^{1/p}\\
& \leq C'r\left(\frac{1}{\overline m_i(2\lambda B)}\int_{2\lambda B}|\Lip  \tilde u(\tilde z)|^p d\overline m_i \right)^{1/p}.
\end{align*}
The switch from $C$ to $C'$ is to call attention to the comparability constant that is used to change from $m_i(\tilde z)$ to $\overline m_i(B(\tilde z, \epsilon_{V_i}))$, and to call attention to the doubling of $\lambda$ in our integral.

 We know from one-dimensional calculus that, on edges, linear functions have the smallest $p-$energy integrals amongst all Sobolev functions with the same boundary values. That is, since $\tilde u$ is a $p$-harmonic function on each individual edge,
\begin{equation} \label{energy}
\int_I |\Lip \tilde u|^p d\overline m_i \leq \int_I |\Lip u|^pd \overline m_i
\end{equation}
whenever $I$ is an edge connecting two points in $V_i$. {To avoid confusion}, we will distinguish the average value of $\tilde u$ on $B$ with respect to $V_i$ and $G_i$ as follows:
\[
\tilde u_B := \frac1{m_i(B)}\sum_{\tilde x \in B} \tilde u(\tilde x)\, m_i(\tilde x) \hspace{.5in} \overline u_B : = \frac1{\overline m_i(B)} \int_B \tilde u(\tilde x)\, d\overline m_i(\tilde x).
\]
Thus, we have that
\begin{equation}
\frac{1}{m_i(B)}\sum_{\tilde{z} \in B} |\tilde u(\tilde{z}) - \tilde u_B|m_i(\tilde z) \leq Cr\left(\avint_{2\lambda B} |\Lip u|^p d\overline m_i\right)^{1/p},
\label{half}
\end{equation}
for some constant $C$. {We may assume that $\tilde u_B = 0$ as subtracting a constant does not change the upper gradient}. Let $\tilde z \sim \tilde w$ in $2B$, and let $\Gamma_{\tilde z \tilde w}$ be the edge connected $\tilde z$ to $\tilde w$. Since $\tilde u$ is a linear function, we see that
\[
\int_{\Gamma_{\tilde z \tilde w}} |\tilde u| d\overline m_i \leq \frac{|\tilde u(\tilde z)| + |\tilde u(\tilde w)|}2\, \overline m_i(\Gamma_{\tilde z \tilde w}).
\]
Then we see that
\[
\int_B|\tilde u|d\overline m_i \leq \sum_{\tilde z \in 2B} |\tilde u(\tilde z)|\overline m_i(B(\tilde z, \epsilon_{V_i})).
\]
Recalling that $\overline m_i(B)$ is comparable to $m_i(B)$ {and (\ref{BltPI})}, we arrive at
\begin{equation}\label{part1}
\avint_B|\tilde u - \overline u_B|d\overline m_i \leq Cr\left(\avint_{2\lambda B} |\Lip u|^pd\overline m_i \right)^{1/p}.
\end{equation}

Now we look at the second term of the right hand side of (\ref{split}):
\[
\avint_B|f - f_B|d\overline m_i.
\]
{First we see that
\[
\int_B |f - f_B| d\overline m_i = \sum_{I\cap B \neq \emptyset} \int_I |f - f_B|d\overline m_i,
\]
where $I$ is an edge in $G_i$.   Recalling that $f = 0$ on $V_i$, and using the same argument as in Case 1 on each of these integrals, we easily find that on each edge $I$ with $\tilde z$ as one of its endpoints
\[
\int_I  |f-f_B| d\overline m_i \leq 2 \int_I |f|d\overline m_i = 2 \int_I |f - f(\tilde z)| d\overline m_i \leq 2 \epsilon_{V_i} \int_I |\Lip f| d\overline m _i.
\]
Summing up over all the intervals, taking averages, and noting that $r > \epsilon_{V_i}$, we have
\[ \avint_B |f-f_B | d\overline m_i \leq 2 r \avint_B | \Lip f | d\overline m_i.\]
Hence,
\[
\avint_B |f - f_B| d\overline m_i \leq 2 r\avint_B|\Lip f|d\overline m_i \leq 2 r\left(\avint_B |\Lip \tilde u|d\overline m_i  + \avint_B |\Lip u|d\overline m_i\right).
\]}
Recalling the fact stated in (\ref{energy}), and applying H\"older's inequality, we then have that
\[
\avint_B |f - f_B| d\overline m_i \leq 4r\left(\avint_B |\Lip u|d\overline m_i \right)^{1/p}.
\]
Using {this} as well as the bound from (\ref{part1}) in (\ref{split}), we arrive at our $(1,p)$-Poincar\'e inequality:
\[
\avint_B |u - u_B| d\overline m_i \leq Cr\left(\avint_{\lambda_1 B} |\Lip u|d\overline m_i \right)^{1/p},
\]
where $C$ and $\lambda_1$ are some constants that depend only on the doubling constant of $m_i$ and the data for the discretized Poincar\'e inequality on $V_i$.
}

\end{proof}
\subsection{Bi-Lipschitz change in metrics} We have transformed the sequence of discrete spaces, $(V_i, d_i, m_i)$, to a sequence of connected spaces, $(G_i, \tilde d_{G_i}, \overline m_i)$, that have a doubling measure, and support a $(1,p)$-Poincar\'e inequality with data that depends only on the uniform doubling constant and uniform data of the discrete sequence. We now sketch the rest of our proof. We introduce a change in metric from $(G_i, \tilde d_{G_i})$ to a new metric $(G_i, d_{G_i})$, making use of the assumed bi-Lipschitz equivalence of $d_i$ and $d_X$ on $V_i$. For two points $\tilde x$ and $\tilde y$ in $V_i \subset G_i$, we define $d_{G_i}(\tilde x, \tilde y) : = d_X(x, y)$, where $x$ and $y$ are the associated points of $\tilde x$ and $\tilde y$ in $X$. We then extend $d_X$ to all of $G_i$ in the same manner that we extended $d_i$ to $G_i$. We show that this new metric is bi-Lipschitz equivalent to $\tilde d_{G_i}$ in the next paragraph.  This new metric allows us to finish the proof of Theorem \ref{VtoX}. First, $(G_i, d_{G_i}, \overline{m}_i)$ satisfies a $(1,p)$-Poincar\'e inequality with data that only depends upon the data of $(G_i, \tilde d_{G_i}, \overline{m}_i)$. This can be seen as the Poincar\'e inequality is bi-Lipschitz invariant (see chapter 8 of \cite{HKST}). Second, it is this sequence, $(G_i, d_{G_i}, \overline{m}_i)$ that we will show pointed measured Gromov Hausdorff converges to $(X, d_X, \overline{\mu})$, where $\overline{\mu}$ is a comparable measure to $\mu$ on all balls.  The relative ease of this conversion is the reason behind the switch to $d_{G_i}$.  Finally, by the result of Cheeger and Keith which we have listed above as Theorem  \ref{GHconv}, $(X,d_X,\overline{\mu})$ will carry a Poincar\'e inequality. Since measures which are comparable on balls also result in the comparability of integrals of measurable functions, $(X,d_X,\mu)$ will carry the desired Poincar\'e inequality.

It is not hard to show that $d_{G_i}$ is a metric, but is not clear that $d_{G_i}$ is bi-Lipschitz equivalent to $\tilde d_{G_i}$. It is trivial to see that the the metrics are bi-Lipschitz equivalent when restricted to points on $V_i \subset G_i$, since $d_{G_i} =d_X$, and $\tilde d_{G_i} = d_i$ in this case. A more complicated case is when $\tilde x$ and $\tilde y$ are points on edges in $G_i$. For ease of exposition, we will label the associated vertices of the edges containing these points by: $\tilde x_1, \tilde x_2, \tilde y_1$, and $\tilde y_2$ respectively. First, without loss of generality, we will take
\[
\tilde d_{G_i}(\tilde x, \tilde y) = d_i(\tilde x_1, \tilde y_1) + |\tilde x - \tilde x_1| + |\tilde y - \tilde y_1|.
\]
Recall that $d_{G_i}$ is found by finding a shortest path through $G_i$ from $\tilde x$ to $\tilde y$. Then we see, by the bi-Lipschitz equivalence of $d_X$ and $d_i$, and recalling that $L \geq 1$,
\begin{align*}
d_{G_i}(\tilde x, \tilde y) &\leq  d_X(\tilde x_1, \tilde y_1) + |\tilde x - \tilde x_1| + |\tilde y - \tilde y_1| \\
 &\leq L d_i(\tilde x_1, \tilde y_1) + L|\tilde x - \tilde x_1| + L|\tilde y - \tilde y_1|\\
&= L\tilde d_{G_i}(\tilde x, \tilde y).
\end{align*}
Alternatively, if we have that
\[
 d_{G_i}(\tilde x, \tilde y) = d_X(\tilde x_1, \tilde y_1) + |\tilde x - \tilde x_1| + |\tilde y - \tilde y_1|,
\]
then we see that
\begin{align*}
\tilde d_{G_i}(\tilde x, \tilde y) &\leq  d_i(\tilde x_1, \tilde y_1) + |\tilde x - \tilde x_1| + |\tilde y - \tilde y_1| \\
 &\leq L d_X(\tilde x_1, \tilde y_1) + L|\tilde x - \tilde x_1| + L|\tilde y - \tilde y_1|\\
&= Ld_{G_i}(\tilde x, \tilde y).
\end{align*}
The cases where either $\tilde x$ or $\tilde y$ are vertices are subcases of the above. Indeed, if $\tilde x$ is a vertex, then we may call it $\tilde x_1$ and the term $|\tilde x - \tilde x_1|$ is zero, and the above still holds. Thus, we have bi-Lipschitz equivalence of $d_{G_i}$ and $\tilde d_{G_i}$, with the same bi-Lipschitz constant of $L$. This implies, by the discussion above, that the family of $(G_i, d_{G_i}, \overline{m}_i)$ supports a $(1,p)$-Poincar\'e inequality with uniform data that depends only upon the doubling constant of $m_i$, $L$, and the data from $(G_i, \tilde d_{G_i}, \overline{m}_i)$.

\subsection{Pointed measured Gromov Hausdorff convergence of $(G_i, d_{G_i}, \overline{m}_i)$}
This section will be dedicated to proving:
\begin{lem} \label{Conv}
{ A subsequence of} $(G_i, d_{G_i}, \overline m_i)$ converges in the pointed measured Gromov Hausdorff sense to $(X, d_X, {\overline \mu})$, where ${\overline \mu}$ is comparable to $\mu$.
\end{lem}
As per the discussion in the first paragraph of the previous subsection, the proof of this lemma will finish the proof of Theorem \ref{VtoX}.

\begin{proof}[Proof of Lemma \ref{Conv}]
Let $q$ be a point in $V_1$, and $n_1(q) \in X$ will be called $q$ by an abuse of notation. Because $(V_i, d_{v_i}, m_i)$ is a nested embedding into $X$, then there is a representative $q \in V_i$ for all $i \in \N$. We begin by showing that $(G_i, d_{G_i}, q) \overset{GH}{\to} (X, d_X, q)$. Let $r> 0$ and $0 < \eta < r$ be fixed numbers. For each $G_i$ we introduce the maps $f_i : G_i \to X$ where $\left.f_i\right|_{V_i} = n_i$, and $f_i$ maps points on the edge set to a closest vertex. That is,
\[
f_i(\tilde x) = x_1\in V_i \subset X,
\]
where $|\tilde x - \tilde x_1| \leq \frac{\epsilon_{V_i}}{2}$.
 Whenever $\tilde x$ is not the midpoint of an edge set, then $f_i$ is clearly well defined. {For a point $\tilde x \in G_i$ on an edge with associated vertices $\tilde x_1, \tilde x_2 \in V_i$ such that $d_{G_i} (\tilde x, \tilde x_1) = d_{G_i}(\tilde x, \tilde x_2)$, $f_i$ may be chosen to take $\tilde x$ to either vertex.} It is clear that $f_i$ are independent of $r$ and $\eta$, and may be used for any choice of these numbers. Furthermore, the first requirement of pointed Gromov Hausdorff convergence is trivially satisfied with these maps.

We notice that since we assume that $V_{i+1} \subset V_i$, and $H_i \to 0$, then $\epsilon_{V_i} \to 0$ as $i \to \infty$. Select $i_0$ large enough so that $\epsilon_{V_i} <\frac{\eta}{2L}$ for all $i \geq i_0$. Let $\tilde x, \tilde y \in B_{G_i}(q, r)$ for some $i \geq i_0$. If $\tilde x$ and $\tilde y$ are both vertices, then $d_X(x,y) = d_{G_i}(\tilde x,\tilde y)$, and the second requirement of pointed Gromov Hausdorff convergence is guaranteed trivially. However, if $\tilde x$ and $\tilde y$ are not vertices, then we still see that
\[
|d_X(f_i(\tilde x), f_i(\tilde y)) - d_{G_i}(\tilde x, \tilde y)| < 2L\epsilon_{V_i} < \eta.
\]
For the third requirement, we need to verify that $B_X(q, r - \eta) \subset N_\eta(f_i(B_{G_i}(q, r)))$. This is easily verified since
\[
B_X(q, r - \eta) \subset f_i(B_{G_i}(q, r - \eta + \epsilon_{V_i})) \subset f_i(B_{G_i}(q, r)).
\]
Thus, we see that
\[
(G_i, d_{G_i}, q) \overset{GH}{\to} (X, d_X, q).
\]

Since both $X$ and $G_i$ are separable, then there exists isometric embeddings of each into $\ell ^\infty$, since the vertices of $G_i$ are an embedded subset of $X$, we can require embeddings that are equal when restricted to $V_i \subset X$ and $V_i \subset G_i$. Then we see that
\[
d_H^{\ell^\infty}(\iota(\overline{B}_X(q,r)), \iota_i(\overline{B}_{G_i}(q,r))) < H_i + \epsilon_{V_i},
\]
where $\iota$ and $\iota_i$ are the embeddings of $X$ and $G_i$ into $\ell^\infty$ respectively. We see that $H_i + \epsilon_i$ goes to $0$ as $i \to \infty$. Thus, to verify the pointed measured Gromov Hausdorff convergence of $(G_i, d_{G_i}, \overline{m}_i)$, we only need to verify that $(\iota_i)_\#\overline m_i\lfloor\overline{B}_{G_i}(q_i,r)$ converges in the weak$^*$ sense to a measure that is comparable  on metric balls to $(\iota)_\#\mu\lfloor\overline{B}_X(q,r)$. An application of the Banach-Steinhaus theorem and the Reisz representation theorem guarantee that a subsequence does indeed weak$^*$ converge to some measure $\overline \mu$. Thus, we have that
\[
(G_i, d_{G_i}, \overline{m}_i) \overset{GH}{\to} (X, d_X, \overline{\mu}).
\]
\end{proof}
By the Theorem \ref{Cheeg1}, we know that $(X, d_X, \overline{\mu})$ has $\overline{\mu}$ as a doubling measure with a constant that ultimately depends on $L$ and the uniform doubling constant of the family of $m_i$. We also have, by Theorem \ref{GHconv}, that $(X, d_X, \overline{\mu})$ supports a $(1,p)$-Poincar\'e inequality with data that only depends on $L$, the uniform doubling constant of the family of $m_i$, and the uniform data of the family $(V_i, d_i, m_i)$. Since for all $i \geq i_0$ we have that $\overline m_i$ is comparable to $\mu$, then $\overline{\mu}$ is also comparable to $\mu$ by construction.

We now assert that $(X, d_X, \mu)$ supports a $(1,p)$-Poincar\'e inequality. This can be seen since changing (\ref{pi}) by a comparable measure only gives a different constant $C$ which depends upon the comparability constant of the two measures. In this case, the comparability constant depends upon $L$ and the doubling constant for $m_i$. Hence, $(X, d_X, \mu)$ supports a $(1,p)$-Poincar\'e inequality with data that depends on $L$, the uniform doubling constant of the family of $m_i$, and the uniform {Poincar\'e inequality} data of the family $(V_i, d_i, m_i)$.

\end{document}